\documentclass[reqno]{amsproc}

\usepackage[margin=1in]{geometry}
\usepackage{setspace, fullpage}
\usepackage{amsmath,amsthm,amscd,amssymb,bbm,wasysym,mathrsfs}
\usepackage{latexsym}
\usepackage[colorlinks,citecolor=red,pagebackref,hypertexnames=false]{hyperref}
\DeclareMathOperator{\supp}{supp}
\DeclareMathOperator{\Span}{span}

\numberwithin{equation}{section}
\theoremstyle{plain}
\newtheorem{theorem}{{\bf Theorem}}[section]
\newtheorem{lemma}[theorem]{{\bf Lemma}}
\newtheorem{corollary}[theorem]{Corollary}

\theoremstyle{definition}
\newtheorem{definition}[theorem]{{\bf Definition}}

\theoremstyle{remark}
\newtheorem{remark}[theorem]{Remark}
\numberwithin{equation}{section}

\begin{document}

\title{Polygonal equalities and virtual degeneracy in
$L_{p}$-spaces\footnote{Corresponding Author: Anthony Weston (westona@canisius.edu).}}

\dedicatory{In memory of Bernard Joseph Weston (1927 -- 2012)}

\author{Casey Kelleher}
\address{Department of Mathematics, University of California, Irvine, CA 92697, USA}
\email{clkelleh@math.uci.edu}
\author{Daniel Miller}
\address{Department of Mathematics, Cornell University, Ithaca, NY 14853, USA}
\email{dm635@cornell.edu}
\author{Trenton Osborn}
\address{Department of Mathematics, Baylor University, Waco, TX 76798, USA}
\email{trenton\_osborn@baylor.edu}
\author{Anthony Weston}
\address{Department of Mathematics and Statistics, Canisius College, Buffalo, NY 14208, USA}
\address{Department of Decision Sciences, University of South Africa, PO Box 392, UNISA 0003, South Africa}
\email{westona@canisius.edu}



\begin{abstract}
Suppose $0 < p \leq 2$ and that $(\Omega, \mu)$ is a measure space for which $L_{p}(\Omega, \mu)$
is at least two-dimensional. The central results of this paper provide a complete description of the subsets
of $L_{p}(\Omega, \mu)$ that have strict $p$-negative type. In order to do this we study non-trivial $p$-polygonal
equalities in $L_{p}(\Omega, \mu)$. These are equalities that can, after appropriate rearrangement and
simplification, be expressed in the form
\begin{eqnarray*}
\sum\limits_{j, i = 1}^{n} \alpha_{j} \alpha_{i} {\| z_{j} - z_{i} \|}_{p}^{p} & = & 0
\end{eqnarray*}
where $\{ z_{1}, \ldots, z_{n} \}$ is a subset of $L_{p}(\Omega, \mu)$ and $\alpha_{1}, \ldots, \alpha_{n}$ are
non-zero real numbers that sum to zero. We provide a complete classification of the non-trivial $p$-polygonal
equalities in $L_{p}(\Omega, \mu)$. The cases $p < 2$ and $p = 2$ are substantially different and are treated
separately. The case $p = 1$ generalizes an elegant result of Elsner, Han, Koltracht, Neumann and Zippin.

Another reason for studying non-trivial $p$-polygonal equalities in $L_{p}(\Omega, \mu)$ is due to
the fact that they preclude the existence of certain types of isometry. For example, our techniques show that if $(X,d)$ is a
metric space that has strict $q$-negative type for some $q \geq p$, then: (1) $(X,d)$ is not isometric to any
linear subspace $W$ of $L_{p}(\Omega, \mu)$ that contains a pair of disjointly supported non-zero vectors, and
(2) $(X,d)$ is not isometric to any subset of $L_{p}(\Omega, \mu)$ that has non-empty interior. Furthermore, in
the case $p = 2$, it also follows that $(X,d)$ is not isometric to any affinely dependent subset of
$L_{2}(\Omega, \mu)$. More generally, we show that if $(Y, \rho)$ is a
metric space whose generalized roundness $\wp$ is finite and if $(X,d)$ is a metric space that has strict
$q$-negative type for some $q \geq \wp$, then $(X,d)$ is not isometric to any metric subspace of $(Y, \rho)$
that admits a non-trivial $p_{1}$-polygonal equality for some $p_{1} \in [\wp, q]$. It is notable in all of these
statements that the metric space $(X,d)$ can, for instance, be any ultrametric space. As a result we obtain
new insights into sophisticated embedding theorems of Lemin and Shkarin.

We conclude the paper by constructing some pathological infinite-dimensional linear subspaces of $\ell_{p}$ that
do not have strict $p$-negative type.
\end{abstract}
\maketitle

\section{Introduction}
The starting point for this paper is the following intriguing result of Elsner \textit{et al}.\ \cite[Theorem 2.3]{Els}.

\begin{theorem}\label{elsner}
Let $\{ x_{k} \}_{k=1}^{n}$ and $\{ y_{k} \}_{k=1}^{n}$ be two collections of functions in $L_{1}(\Omega, \mu)$. Then
\begin{eqnarray}\label{els:one}
\sum\limits_{j_{1} < j_{2}} {\| x_{j_{1}} - x_{j_{2}} \|}_{1} +
\sum\limits_{i_{1} < i_{2}} {\| y_{i_{1}} - y_{i_{2}} \|}_{1}
& = &
\sum\limits_{j, i = 1}^{n} {\| x_{j} - y_{i} \|}_{1}
\end{eqnarray}
if and only if for almost every $\omega \in \Omega$, the numerical sets
$\{ x_{k}(\omega) \}_{k=1}^{n}$ and $\{ y_{k}(\omega) \}_{k=1}^{n}$ are identical.
\end{theorem}

It is helpful to recall that a \textit{numerical set} is just a finite collection of possibly repeated
(real or complex) numbers. Two numerical sets
$\{ \zeta_{k} \}_{k=1}^{n}$ and $\{ \xi_{k} \}_{k=1}^{n}$ are said to be \textit{identical} if
there exists a permutation $\pi$ of $(1,2, \ldots, n)$ such that
$\zeta_{\pi(k)} = \xi_{k}$ for each $k, 1 \leq k \leq n$.

Our interest in Theorem \ref{elsner} is that it is related to the problem of
characterizing all cases of non-trivial equality in the $1$-negative type
inequalities for $L_{1}(\Omega, \mu)$ (see Definition \ref{neg:gen} (1)). For instance, if we assume in the statement of
Theorem \ref{elsner} that the functions $x_{1}, \ldots, x_{n}, y_{1}, \ldots, y_{n}$ are
pairwise distinct, then the equality (\ref{els:one}) may be rewritten in the form
\begin{eqnarray}\label{els:two}
\sum\limits_{j,i = 1}^{2n} \alpha_{j} \alpha_{i} {\| z_{j} - z_{i} \|}_{1}
& = & 0
\end{eqnarray}
where the real numbers $\alpha_{1}, \ldots, \alpha_{2n}$ sum to zero with each $\alpha_{k} \in \{ -1, 1\}$
and $z_{k}$ is an $x_{j}$ or $y_{i}$ depending on the sign of $\alpha_{k}$.
In this way we see that Theorem \ref{elsner} provides specific instances
of non-trivial equality in the $1$-negative type inequalities for $L_{1}(\Omega, \mu)$.

In this paper we present an in-depth study of non-trivial $p$-polygonal equalities in $L_{p}$-spaces,
$0 < p < \infty$. These are equalities that can, after appropriate rearrangement and simplification,
be expressed in the form
\begin{eqnarray}\label{nt:eq}
\sum\limits_{j, i = 1}^{n} \alpha_{j} \alpha_{i} {\| z_{j} - z_{i} \|}_{p}^{p} & = & 0
\end{eqnarray}
where $\{ z_{1}, \ldots, z_{n} \}$ is a subset of $L_{p}(\Omega, \mu)$ and
$\alpha_{1}, \ldots, \alpha_{n}$ are non-zero real numbers such that $\alpha_{1} + \cdots + \alpha_{n} = 0$.
Throughout this paper all measures are non-trivial and positive. Moreover, all $L_{p}$-spaces are endowed with the
usual (quasi-) norm and are assumed to be at least two-dimensional.

Equalities of the form (\ref{nt:eq}) especially enlightening if $0 < p \leq 2$. In this case,
$L_{p}(\Omega, \mu)$ has $p$-negative type but it does not have $q$-negative type for any $q > p$.
So if $0 < p \leq 2$ we may view each non-trivial $p$-polygonal equality as being an instance
of non-trivial equality in a $p$-negative type inequality for $L_{p}(\Omega, \mu)$,
and \textit{vice-versa}. Theorems \ref{vd:simplex} and \ref{thm:H} classify all non-trivial
$p$-polygonal equalities in $L_{p}(\Omega, \mu)$, $0 < p \leq 2$.
Corollaries \ref{Lp} and \ref{cor:H1} then classify the subsets of $L_{p}(\Omega, \mu)$, $0 < p \leq 2$,
that have strict $p$-negative type. Our approach in the case $0 < p < 2$ is based on a new property
of $L_{p}$-spaces that we call virtual degeneracy. Specialization to the case $p=1$ reveals a more
general form of Theorem \ref{elsner}. The Hilbert space case $p = 2$ is notably different on account
of the parallelogram identity and is thus treated separately. In the case $p > 2$
we obtain partial results about the collection of all non-trivial $p$-polygonal equalities in
$L_{p}(\Omega, \mu)$.

Another reason for studying $p$-polygonal equalities in $L_{p}$-spaces is that
they preclude the existence of certain types of isometry. Theorem \ref{embed} shows that if $0 < p < \infty$ and if $(X,d)$ is a
metric space that has strict $q$-negative type for some $q \geq p$, then: (1) $(X,d)$ is not isometric to any linear subspace $W$
of $L_{p}(\Omega, \mu)$ that contains a pair of disjointly supported non-zero vectors, and
(2) $(X,d)$ is not isometric to any subset of $L_{p}(\Omega, \mu)$ that has non-empty interior. Furthermore, in
the case $p = 2$, it follows that we also have: (3) $(X,d)$ is not isometric to any affinely dependent subset of $L_{2}(\Omega, \mu)$.
These theorems are instances of a more general embedding principle (Theorem \ref{iso:emb2}):
If $(Y, \rho)$ is a metric space whose generalized roundness $\wp$ (see Definition \ref{neg:gen} (3))
is finite and if $(X,d)$ is a metric space that has strict $q$-negative type for some $q \geq \wp$, then
$(X,d)$ is not isometric to any metric subspace of $(Y, \rho)$ that admits a non-trivial $p_{1}$-polygonal
equality for some $p_{1} \in [\wp, q]$. It is notable in all of these instances
that the metric space $(X,d)$ can, for example, be any ultrametric space. This provides new insights into
embedding theorems of Lemin \cite{Lem} and Shkarin \cite{Shk}.

As a basic technique used throughout the paper we do not deal directly with equalities of the form (\ref{nt:eq}).
We instead work with non-trivial weighted generalized roundness equalities with exponent $p$. These
are an equivalent family of equalities that can, after rearrangement and simplification, be expressed in the form
\begin{eqnarray}\label{gr:eq}
\sum\limits_{j_{1} < j_{2}} m_{j_{1}}m_{j_{2}} {\| x_{j_{1}} - x_{j_{2}}\|}_{p}^{p} +
\sum\limits_{i_{1} < i_{2}} n_{i_{1}}n_{i_{2}} {\| y_{i_{1}} - y_{i_{2}}\|}_{p}^{p}
&   =  & \sum\limits_{j, i} m_{j}n_{i} {\| x_{j} - y_{i} \|}_{p}^{p}.
\end{eqnarray}
where $x_{1}, \ldots, x_{s}, y_{1}, \ldots, y_{t} \in L_{p}(\Omega, \mu)$,
$\{ x_{j} : 1 \leq j \leq s \} \cap \{ y_{i} : 1 \leq i \leq t \} = \varnothing$,
$m_{1}, \ldots, m_{s} > 0$, $n_{1}, \ldots, n_{t} > 0$, and $m_{1} + \cdots + m_{s} = n_{1} + \cdots + n_{t}$.
We derive the precise transition between equalities of the form (\ref{nt:eq}) and (\ref{gr:eq}) in Lemma \ref{SL1}
and Theorem \ref{2.3}.

The organization of the paper is as follows. In Section \ref{sec:2} we recall the notions
of negative type and generalized roundness. In Section \ref{sec:3} we introduce the notion of non-trivial $p$-polygonal
equalities in metric spaces and examine their bearing on the existence of isometries. In Section \ref{sec:4} we focus
on non-trivial $p$-polygonal equalities in $L_{p}$-spaces ($p \not= 2$) and introduce the notion of virtual degeneracy.
All results in Section \ref{sec:4} are equally valid for real or complex $L_{p}$-spaces.
Section \ref{sec:5} is dedicated to studying $2$-polygonal equalities in real or complex inner product and Hilbert spaces.
Section \ref{sec:6} considers virtually degenerate linear subspaces of
$L_{p}$-spaces, $0 < p < \infty$. Such linear subspaces do not have $q$-negative
type for any $q > p$. Lemma \ref{vds:lem} and Theorem \ref{inf:vds} provide ways to construct
virtually degenerate linear subspaces of $L_{p}(\Omega, \mu)$, $0 < p < \infty$.
Section \ref{sec:7} completes the paper with a discussion of open problems.
Throughout the paper we use $\mathbb{N}$ to denote the set of all positive integers.
Whenever sums are indexed over the empty set they are defined to be $0$.

\section{Negative type and generalized roundness}\label{sec:2}
The notions of negative type and generalized roundness recalled below in Definition \ref{neg:gen}
were formally introduced and studied by Menger \cite{Men}, Schoenberg \cite{Sc1, Sc2, Sc3} and Enflo \cite{En2}.
Menger and Schoenberg were interested in determining which metric spaces can be isometrically
embedded into a Hilbert space. Enflo's interest, on the other hand, was to construct a
separable metric space that admits no uniform embedding into any Hilbert space.
More recently, there has been interest in the notion of ``strict'' $p$-negative type, particularly
as it pertains to the geometry of finite metric spaces. In the present work we will see that
strict $p$-negative type can also have a role to play in certain infinite-dimensional settings.
Papers that have been instrumental in developing properties of metrics of strict $p$-negative type
include \cite{Hj1, Hj2, Dou, Hli, Wol, San}.

\begin{definition}\label{neg:gen} Let $p \geq 0$ and let $(X,d)$ be a metric space.
\begin{enumerate}
\item $(X,d)$ has $p$-{\textit{negative type}} if and only if for
all finite subsets $\{z_{1}, \ldots , z_{n} \} \subseteq X$ and all scalar $n$-tuples
$\boldsymbol{\alpha} = (\alpha_{1}, \ldots, \alpha_{n}) \in \mathbb{R}^{n}$ that satisfy
$\alpha_{1} + \cdots + \alpha_{n} = 0$, we have:
\begin{eqnarray}\label{p:neg}
\sum\limits_{j,i =1}^{n} d(z_{j},z_{i})^{p} \alpha_{j} \alpha_{i} & \leq & 0.
\end{eqnarray}
As a notational aid we set $\Pi_{0} = \{ \boldsymbol{\alpha} \in \mathbb{R}^{n} : \alpha_{1} + \cdots + \alpha_{n} = 0 \}$.

\item $(X,d)$ has \textit{strict} $p$-{\textit{negative type}} if and only if it has $p$-negative type
and the inequalities (\ref{p:neg}) are strict except in the trivial case $\boldsymbol{\alpha} = \boldsymbol{0}$.

\item The \textit{generalized roundness} (or \textit{supremal negative type}) of
$(X,d)$ is defined to be the quantity $\wp(X) = \sup \{ q : (X,d) \text{ has } q\text{-negative type} \}$.
\end{enumerate}
\end{definition}

It is worth noting that if we set $D_{p} = (d(z_{j}, z_{i})^{p})_{j,i}$ and let $\langle \cdot, \cdot\rangle$ denote the usual inner
product on $\mathbb{R}^{n}$, the inequalities (\ref{p:neg}) may be expressed more succinctly as:
$\langle D_{p}\boldsymbol{\alpha}, \boldsymbol{\alpha} \rangle \leq 0$ for all $\boldsymbol{\alpha} \in \Pi_{0}$.

Negative type holds on closed intervals by a result of Schoenberg \cite[Theorem 2]{Sc2}.
Indeed, the set of all values of $p$ for which a given metric space $(X,d)$ has $p$-negative type is always an interval of
the form $[0, \wp(X)]$ or $[0, \infty)$. We allow the case $\wp(X) = 0$, in which case the interval
degenerates to $\{ 0 \}$. Li and Weston \cite{Hli} have obtained a version of Schoenberg's theorem that deals with strict
negative type.

\begin{theorem}[Li and Weston \cite{Hli}]\label{LW1}
Let $(X,d)$ be a metric space. If $(X,d)$ has $p$-negative type for some $p > 0$, then
$(X,d)$ has strict $q$-negative type for all $q$ such that $0 \leq q < p$.
\end{theorem}

In the case of finite metric spaces a more definitive statement holds.

\begin{theorem}[Li and Weston \cite{Hli}]\label{LW2}
A finite metric space $(X,d)$ has strict $p$-negative type if and only if $p < \wp(X)$.
\end{theorem}

Theorem \ref{LW2} is specific to finite metric spaces. The supremal negative type of
an infinite metric space may or may not be strict. This may be seen from examples in \cite{Dou} and \cite{Hli}.

It is important to note that if a metric space $(X,d)$ has $\wp$-negative type but not strict
$\wp$-negative type for some $\wp > 0$, then $\wp = \wp(X)$ as a corollary of Theorem \ref{LW1}.
Moreover, under these conditions, there must exist a subset $\{ z_{1}, \ldots, z_{n} \} \subseteq X$
and an $n$-tuple $\boldsymbol{\alpha} = (\alpha_{1}, \ldots, \alpha_{n}) \in \Pi_{0} \setminus \{ \boldsymbol{0} \}$
such that
\begin{eqnarray}\label{p:negzero}
\sum\limits_{j,i =1}^{n} d(z_{j},z_{i})^{\wp} \alpha_{j} \alpha_{i} & = & 0.
\end{eqnarray}
Clearly, if some $\alpha_{k} = 0$ in this setting, then we may discard the pair $(z_{k}, \alpha_{k})$
from the configuration without disrupting the underlying equalities. In such situations, we may
therefore assume that every $\alpha_{k}$ is non-zero.

The following fundamental fact is central to the development of the rest of this paper.

\begin{theorem}[Schoenberg \cite{Sc3}]\label{schoenberg}
Let $0 < p \leq 2$ and suppose that $(\Omega, \mu)$ is a measure space. Then $L_{p}(\Omega, \mu)$
has $p$-negative type but it does not have $q$-negative type for any $q > p$. In other words, $\wp(L_{p}) = p$.
\end{theorem}

\begin{remark}\label{Kold}
In the statement of Theorem \ref{schoenberg} we are assuming that $L_{p}(\Omega, \mu)$ is at least
$2$-dimensional. It is notable that if $2 < p \leq \infty$ and if $L_{p}(\Omega, \mu)$ is at least
$3$-dimensional, then $L_{p}(\Omega, \mu)$ does not have $q$-negative type for any $q > 0$. This follows from
theorems of Dor \cite{Dor}, Misiewicz \cite{Mis} and Koldobsky \cite{Kol}.
\end{remark}


\section{Polygonal equalities in metric spaces}\label{sec:3}

In order to address cases of equality in (\ref{p:neg}) it is helpful to reformulate
Definition \ref{neg:gen} in terms of signed $(s,t)$-simplices and the corresponding $p$-simplex ``gaps,''
the notions of which we now introduce.

\begin{definition}\label{S1}
Let $X$ be a set and suppose that $s,t > 0$ are integers.
A \textit{signed $(s,t)$-simplex in $X$} is a collection of (not necessarily distinct) points
$x_{1}, \ldots, x_{s}, y_{1}, \ldots, y_{t} \in X$ together with a corresponding collection
of real numbers $m_{1}, \ldots, m_{s}, n_{1}, \ldots, n_{t}$ that satisfy $m_{1} + \cdots + m_{s}
= n_{1} + \cdots + n_{t}$. Such a configuration of points and real numbers will be denoted by
$D = [x_{j}(m_{j});y_{i}(n_{i})]_{s,t}$ and will simply be called a \textit{simplex} when
no confusion can arise.
\end{definition}

Simplices with weights on the vertices were introduced by Weston \cite{We1} to study the generalized
roundness of finite metric spaces. In \cite{We1} the author only considers positive weights. The
approach being taken here appears to be more general but it is, in fact, equivalent by Lemma \ref{SL1}.
The basis for the following definition is derived from the original formulation of generalized
roundness that was introduced by Enflo \cite{En2} in order to address a problem of Smirnov concerning
the uniform structure of Hilbert spaces.

\begin{definition}\label{S4}
Let $(X,d)$ be a metric space and suppose that $p$ is a non-negative real number.
For each signed $(s,t)$-simplex $D = [x_{j}(m_{j});y_{i}(n_{i})]_{s,t}$ in $X$ we define
\begin{eqnarray*}
\gamma_{p}(D) & = &
\sum\limits_{j,i = 1}^{s,t} m_{j}n_{i}d(x_{j},y_{i})^{p} -
\sum\limits_{ 1 \leq j_{1} < j_{2} \leq s} m_{j_{1}}m_{j_{2}}d(x_{j_{1}},x_{j_{2}})^{p} -
\sum\limits_{ 1 \leq i_{1} < i_{2} \leq t} n_{i_{1}}n_{i_{2}}d(y_{i_{1}},y_{i_{2}})^{p}.
\end{eqnarray*}
We call $\gamma_{p}(D)$ the \textit{$p$-simplex gap of $D$ in $(X,d)$}.
\end{definition}

In the formulation of Definition \ref{S1} the points $x_{1}, \ldots, x_{s}, y_{1}, \ldots, y_{t} \in X$ are not
required to be distinct. As we will see, this allows a high degree of flexibility but it also comes at some technical
cost, the most important of which is the necessity of keeping track of repetitions of points using repeating numbers.

\begin{definition}\label{S3}
Given a signed $(s,t)$-simplex $D = [x_{j}(m_{j});y_{i}(n_{i})]_{s,t}$ in a set $X$ we
denote by $S(D)$ the set of distinct points in $X$ that appear in $D$. In other words,
$$S(D) = \{ z \in X: z = x_{j} \text{ for some } j \text{ or } z = y_{i} \text{ for some } i \}.$$
For each $z \in S(D)$ we define the \textit{repeating numbers} ${\mathbf{m}}(z)$ and ${\mathbf{n}}(z)$ as follows:
\begin{eqnarray*}
{\mathbf{m}}(z) & = & \sum\limits_{j : z = x_{j}} m_{j}, \mbox{ and} \\
{\mathbf{n}}(z) & = & \sum\limits_{i : z = y_{i}} n_{i}.
\end{eqnarray*}
We say that the simplex $D$ is \textit{degenerate} if ${\mathbf{m}}(z) = {\mathbf{n}}(z)$ for all $z \in S(D)$.
\end{definition}

In relation to Definition \ref{S3} it is important to note that some of the sums defining ${\mathbf{m}}(z)$
or ${\mathbf{n}}(z)$, $z \in S(D)$, may be indexed over the empty set $\varnothing$. The convention in this paper is
that all such sums are equal to $0$. By Definition \ref{S1}, $\sum {\mathbf{m}}(z) = \sum {\mathbf{n}}(z)$.

\begin{remark}\label{snifter}
It is worth noting that if $D = [x_{j}(m_{j});y_{i}(n_{i})]_{s,t}$ is a degenerate signed $(s,t)$-simplex in a
vector space $X$, then
\begin{eqnarray*}
\sum\limits_{j = 1}^{s} m_{j}x_{j} & = & \sum\limits_{z \in S(D)} {\mathbf{m}}(z) z \\
                                   & = & \sum\limits_{z \in S(D)} {\mathbf{n}}(z) z \\
                                   & = & \sum\limits_{i = 1}^{t} n_{i}y_{i}.
\end{eqnarray*}
The full significance of this remark will become apparent as we proceed.
\end{remark}

There are various ways that we may refine a signed $(s,t)$-simplex $D = [x_{j}(m_{j});y_{i}(n_{i})]_{s,t}$
in a metric space $(X,d)$ without altering any of the values of the $p$-simplex gaps $\gamma_{p}(D)$, $p > 0$.
The following lemmas describe three such scenarios. The first lemma is particularly simple and is stated
without proof.

\begin{lemma}[Procedure 1]\label{Proc1} Suppose $D = [x_{j}(m_{j});y_{i}(n_{i})]_{s,t}$
is a signed $(s,t)$-simplex in a metric space $(X,d)$ such that $x_{1} = x_{2}$. Let
$D^{\prime} = [x_{2}(m_{1}+m_{2}), x_{3}(m_{3}), \ldots, x_{s}(m_{s}); y_{i}(n_{i})]_{s-1,t}$.
Then $S(D) = S(D^{\prime})$ and $\gamma_{p}(D) = \gamma_{p}(D^{\prime})$ for all $p \geq 0$. Moreover,
all repeating numbers are invariant under this refinement.
\end{lemma}

\begin{lemma}[Procedure 2]\label{Proc2} Suppose $D = [x_{j}(m_{j});y_{i}(n_{i})]_{s,t}$
is a signed $(s,t)$-simplex in a metric space $(X,d)$ such that $x_{1} = y_{1}$.
Let $D^{\prime\prime} = [x_{1}(0), x_{2}(m_{2}), \ldots, x_{s}(m_{s}) ;
y_{1}(n_{1}-m_{1}), y_{2}(n_{2}), \ldots, y_{t}(n_{t})]_{s,t}$.
Then $S(D) = S(D^{\prime\prime})$ and $\gamma_{p}(D) = \gamma_{p}(D^{\prime\prime})$ for all $p \geq 0$.
Moreover, as $x_{1} = y_{1} = z$ for some $z \in S(D)$, it is also the case that
$\mathbf{m}_{D^{\prime\prime}}(z) = \mathbf{m}_{D}(z) - m_{1}$ and $\mathbf{n}_{D^{\prime\prime}}(z)
= \mathbf{n}_{D}(z) - m_{1}$.
\end{lemma}

\begin{proof}
Let $p \geq 0$ be given.
The assumption is that $x_{1} = y_{1}$. Let $\gamma_{p}(D)(x_{1})$ denote the contribution of the $x_{1}$-terms
to the $p$-simplex gap $\gamma_{p}(D)$. Similarly, let $\gamma_{p}(D)(y_{1})$ and $\gamma_{p}(D^{\prime\prime})(y_{1})$
denote the contribution of the $y_{1}$-terms to $\gamma_{p}(D)$ and $\gamma_{p}(D^{\prime\prime})$, respectively.
It suffices to show that $\gamma_{p}(D)(x_{1}) + \gamma_{p}(D)(y_{1}) = \gamma_{p}(D^{\prime\prime})(y_{1})$.
(This just says that ``$x_{1}(m_{1})$'' cancels ``$y_{1}(m_{1})$'' in $\gamma_{p}(D)$, leaving $\gamma_{p}(D^{\prime\prime})$.)
First of all, note that
\begin{eqnarray*}
\gamma_{p}(D)(x_{1})
& = & m_{1} \biggl( \sum\limits_{i=2}^{t} n_{i}d(x_{1}, y_{i})^{p} \biggl) -
      m_{1} \biggl( \sum\limits_{j=2}^{s} m_{j}d(x_{1}, x_{j})^{p} \biggl) \\
& = & m_{1} \biggl( \sum\limits_{i=2}^{t} n_{i}d(y_{1}, y_{i})^{p} \biggl) -
      m_{1} \biggl( \sum\limits_{j=2}^{s} m_{j}d(x_{1}, x_{j})^{p} \biggl).
\end{eqnarray*}
On the other hand,
\begin{eqnarray*}
\gamma_{p}(D)(y_{1})
& = & (n_{1} - m_{1} + m_{1}) \biggl( \sum\limits_{j=2}^{s} m_{j}d(x_{j}, y_{1})^{p}\biggl) -
      (n_{1} - m_{1} + m_{1}) \biggl( \sum\limits_{i=2}^{t} n_{i}d(y_{1}, y_{i})^{p}\biggl) \\
& = & (n_{1} - m_{1}) \biggl( \sum\limits_{j=2}^{s} m_{j}d(x_{j}, y_{1})^{p}\biggl) -
      (n_{1} - m_{1}) \biggl( \sum\limits_{i=2}^{t} n_{i}d(y_{1}, y_{i})^{p}\biggl) \\
&   & + m_{1} \biggl( \sum\limits_{j=2}^{s} m_{j}d(x_{j}, y_{1})^{p}\biggl) -
      m_{1} \biggl( \sum\limits_{i=2}^{t} n_{i}d(y_{1}, y_{i})^{p}\biggl) \\
& = & \gamma_{p}(D^{\prime\prime})(y_{1}) - \gamma_{p}(D)(x_{1}).
\end{eqnarray*}
The second assertion of the lemma is true by construction.
\end{proof}

\textbf{Notation.} The refinement procedure used to form the simplex $D^{\prime\prime}$ in the statement of Lemma \ref{Proc2}
will be denoted: $x_{1}(m_{1}) \rightarrow x_{1}(0), y_{1}(n_{1}) \rightarrow y_{1}(n_{1} - m_{1})$.

There is a useful variant of the second refinement procedure.

\begin{lemma}[Procedure 3]\label{Proc3} Suppose that $D = [x_{j}(m_{j});y_{i}(n_{i})]_{s,t}$
is a signed $(s,t)$-simplex in a metric space $(X,d)$. Let $y_{t+1} = x_{1}$, $n_{t+1} = - m_{1}$
and $D^{\prime\prime\prime} = [x_{1}(0), x_{2}(m_{2}), \ldots, x_{s}(m_{s}) ; y_{i}(n_{i})]_{s,t+1}$.
Then $S(D) = S(D^{\prime\prime\prime})$ and $\gamma_{p}(D) = \gamma_{p}(D^{\prime\prime\prime})$ for all $p \geq 0$.
Moreover, as $x_{1} = y_{t+1} = z$ for some $z \in S(D)$, it is also the case that
$\mathbf{m}_{D^{\prime\prime\prime}}(z) = \mathbf{m}_{D}(z) - m_{1}$ and
$\mathbf{n}_{D^{\prime\prime\prime}}(z) = \mathbf{n}_{D}(z) - m_{1}$.
\end{lemma}

\begin{proof}
If we set $y_{t+1} = x_{1}$, then we may insert the pair $y_{t+1}(0)$ into the simplex $D$ without altering
any of the simplex gaps or repeating numbers. Now apply Procedure 2 to refine the simplex $D$ in the following way:
$x_{1}(m_{1}) \rightarrow x_{1}(0), y_{t+1}(0) \rightarrow y_{t+1}(-m_{1})$.
\end{proof}

\textbf{Notation.} The refinement procedure used to form the simplex $D^{\prime\prime\prime}$ in the statement of Lemma \ref{Proc3}
will be denoted: $x_{1}(m_{1}) \rightarrow x_{1}(0), y_{t+1}(n_{t+1}) = x_{1}(-m_{1})$.

It is worth noting that the refinement procedures described in Lemmas \ref{Proc1} -- \ref{Proc3} preserve
degeneracy: Any refinement of a degenerate simplex will also be degenerate.
More generally, we may use the refinement procedures to define a preorder on the
collection of all signed simplices in a metric space $(X,d)$.

\begin{definition}\label{S5}
Let $D_{1}, D_{2}$ be signed simplices in a metric space $(X,d)$. We say that $D_{1}$ \textit{refines to}
$D_{2}$, denoted $D_{1} \succ D_{2}$,  if $D_{2}$ can be obtained from $D_{1}$ by finitely many applications
of the three refinement procedures described in Lemmas \ref{Proc1} -- \ref{Proc3}.
\end{definition}

\begin{remark}\label{SR5} Notice that if $D_{1} \succ D_{2}$ in a metric space $(X,d)$, then:
\begin{enumerate}
\item[(1)] $S(D_{1}) = S(D_{2})$,

\item[(2)] for each $z \in S(D_{1}) = S(D_{2})$, ${\mathbf{m}}_{D_{1}}(z) = {\mathbf{n}}_{D_{1}}(z)$ if and only if
${\mathbf{m}}_{D_{2}}(z)= {\mathbf{n}}_{D_{2}}(z)$, and 

\item[(3)] $\gamma_{p}(D_{1}) = \gamma_{p}(D_{2})$ for all $p \geq 0$.
\end{enumerate}
Moreover, it follows from (2), that $D_{1}$ is degenerate if and only if $D_{2}$ is degenerate.
\end{remark}

The notion of refining to another simplex affords an important characterization of degeneracy.

\begin{lemma}\label{deg:s}
A signed $(s,t)$-simplex $D = [x_{j}(m_{j});y_{i}(n_{i})]_{s,t}$ in a metric space $(X,d)$
is degenerate if and only if it refines to a signed simplex $D_{\varnothing}$ that has
no non-zero weights.
\end{lemma}

\begin{proof}
$(\Rightarrow)$ Suppose that $D = [x_{j}(m_{j});y_{i}(n_{i})]_{s,t}$ is a given degenerate simplex.
We may assume, by applying Lemma \ref{Proc1} finitely often if necessary, that the degenerate simplex $D$
has the following property: the points $x_{1}, \ldots, x_{s}$ are pairwise distinct and the points
$y_{1}, \ldots, y_{t}$ are pairwise distinct. (This forces $s = t$ because $D$ is degenerate.)
Let $z \in S(D)$ be given. By assumption, $\mathbf{m}_{D}(z) = \mathbf{n}_{D}(z)$. By relabeling the
simplex, if necessary, we may assume that $z = x_{1} = y_{1}$.
The property placed on $D$ then ensures that $m_{1} = \mathbf{m}_{D}(z) =
\mathbf{n}_{D}(z) = n_{1}$. By Lemma \ref{Proc2}, it follows that $D$ refines to the simplex
$D^{\prime\prime} = [x_{1}(0), x_{2}(m_{2}), \ldots, x_{s}(m_{s}); y_{1}(0), y_{2}(n_{2}), \ldots, y_{t}(n_{t})]_{s,t}$.
The forward implication of the lemma follows by applying this process a finite number of times.

$(\Leftarrow)$ Immediate from Remark \ref{SR5}.
\end{proof}

As an immediate application of Lemma \ref{deg:s} we obtain the following corollary.

\begin{corollary}\label{0}
If $D$ is a degenerate simplex in a metric space $(X,d)$, then $\gamma_{p}(D) = 0$ for all $p \geq 0$.
\end{corollary}

In fact, the converse of Corollary \ref{0} also holds, but this will require an appeal to Theorem \ref{2.3}.

At this point it is helpful to introduce some additional descriptive terminology for simplices.

\begin{definition}\label{S2}
Let $D = [x_{j}(m_{j});y_{i}(n_{i})]_{s,t}$ be a signed $(s,t)$-simplex in a set $X$.
\begin{enumerate}
\item[(1)] $D$ is said to be \textit{pure} if $x_{j} \not= y_{i}$ for all $j,i$.

\item[(2)] $D$ is said to be \textit{full} if
the points $x_{1}, \ldots, x_{s} \in X$ are pairwise distinct and the points
$y_{1}, \ldots, y_{t} \in X$ are pairwise distinct.

\item[(3)] $D$ is \textit{completely refined} if it is full, pure and every weight is positive ($> 0$).
\end{enumerate}
\end{definition}

The following lemma presents a fundamental dichotomy for a simplex in a metric space.

\begin{lemma}\label{SL1}
If $D = [x_{j}(m_{j});y_{i}(n_{i})]_{s,t}$ is a signed $(s,t)$-simplex in a metric space
$(X,d)$, then exactly one of the following two statements must hold:
\begin{enumerate}
\item[(1)] $D$ is degenerate.

\item[(2)] $D$ refines to a completely refined simplex $D^{\ast\ast\ast}$.
\end{enumerate}
\end{lemma}

\begin{proof}
Let $D = [x_{j}(m_{j});y_{i}(n_{i})]_{s,t}$ be a given signed $(s,t)$-simplex in $(X,d)$.
There is, by finitely many of applications of Lemma \ref{Proc1}, a full simplex $D^{\prime}$ such that
$D \succ D^{\prime}$, $S(D) = S(D^{\prime})$ and $D^{\prime}$ has identical repeating numbers to $D$.
So we may as well assume from the
outset that the given simplex $D$ is full. There are five ways that we may then choose to refine the full simplex $D$ to
form a simplex $D^{\ast}$ such that $D \succ D^{\ast}$:

If $x_{j} = y_{i}$ for some $j,i$, then Lemma \ref{Proc2} allows us to form a
simplex $D^{\ast}$ by implementing the appropriate refinement from the following list:

\begin{enumerate}
\item[(1)] If $m_{j} = n_{i}$: $x_{j}(m_{j}) \rightarrow x_{j}(0), y_{i}(n_{i}) \rightarrow y_{i}(0)$.

\item[(2)] If $m_{j} < n_{i}$: $x_{j}(m_{j}) \rightarrow x_{j}(0), y_{i}(n_{i}) \rightarrow y_{i}(n_{i} - m_{j})$.

\item[(3)] If $m_{j} > n_{i}$: $x_{j}(m_{j}) \rightarrow x_{j}(m_{j} - n_{i}), y_{i}(n_{i}) \rightarrow y_{i}(0)$.
\end{enumerate}

If there is a $j$ such that $x_{j} \not= y_{i}$ for all $i$ and $m_{j} < 0$, then
Lemma \ref{Proc3} allows us to form a simplex $D^{\ast}$
by implementing the following refinement:

\begin{enumerate}
\item[(4)] $x_{j}(m_{j}) \rightarrow x_{j}(0), y_{t+1}(n_{t+1}) = x_{j}(-m_{j})$.
\end{enumerate}

If there is an $i$ such that $y_{i} \not= x_{j}$ for all $j$ and $n_{i} < 0$, then
Lemma \ref{Proc3} allows us to form a simplex $D^{\ast}$
by implementing the following refinement:

\begin{enumerate}
\item[(5)] $x_{s+1}(m_{s+1}) =  y_{i}(-n_{i}), y_{i}(n_{i}) \rightarrow y_{i}(0)$.
\end{enumerate}

We now proceed algorithmically. Consider the least $j$ such that ($x_{j} = y_{i}$ for some
$i$) $\vee$ ($x_{j} \not= y_{i}$ for all $i$ and $m_{j} < 0$), and implement the appropriate refinement
from the list (1) -- (4). Reiterate this process a finite number of times until no such $j$ remain.
Then consider the least $i$ such that $y_{i} \not= x_{j}$ for all $j$ and $n_{i} < 0$, and implement
the refinement (5). Reiterate this process a finite number of times until no such $i$ remain. At
this point the algorithm terminates and outputs a simplex $D^{\ast\ast}$ such that $D \succ D^{\ast\ast}$.
By construction, each vertex in the simplex $D^{\ast\ast}$
has non-negative weight. Moreover, $D$ is degenerate if all vertices in $D^{\ast\ast}$ have weight $0$.
On the other hand, $D$ is non-degenerate if at least one vertex in $D^{\ast\ast}$ has positive weight.
In the latter case, by deleting all vertices from $D^{\ast\ast}$ that have weight $0$, it
follows that $D$ refines to a completely refined simplex $D^{\ast\ast\ast}$.
\end{proof}

\begin{remark}
If $D = [x_{j}(m_{j});y_{i}(n_{i})]_{s,t}$ is a non-degenerate signed $(s,t)$-simplex in a metric
space $(X,d)$, then the completely refined signed simplex that it reduces to is unique (modulo
relabeling), and we can give it explicitly. To do this, let $\{ \tilde{x}_{j} \}_{\tilde{s}} =
\{ z \in S(D) : \mathbf{m}(z) > \mathbf{n}(z) \}$, $\{ \tilde{y}_{i} \}_{\tilde{t}} =
\{ z \in S(D) : \mathbf{m}(z) < \mathbf{n}(z) \}$, $\tilde{m}_{j} = \mathbf{m}(\tilde{x}_{j}) - \mathbf{n}(\tilde{x}_{j})$
and $\tilde{n}_{i} = \mathbf{n}(\tilde{y}_{i}) - \mathbf{m}(\tilde{y}_{i})$.
Then $D^{\ast\ast\ast} = [\tilde{x}_{j}(\tilde{m}_{j}); \tilde{y}_{i}(\tilde{n}_{i})]_{\tilde{s}, \tilde{t}}$.
If, moreover, $\mathcal{N}_{(X,d)}$ denotes the set of all non-degenerate simplices that correspond to the metric
space $(X,d)$, then the operation $\succ$ induces an equivalence relation on $\mathcal{N}_{(X,d)}$:
$D_{1} \sim D_{2}$ if and only if $D_{1}$ and $D_{2}$ refine to a common completely refined
simplex $D^{\ast\ast\ast}$.
\end{remark}

In the case of vector spaces we will have reason to consider balanced simplices.

\begin{definition}\label{S2B}
Let $D = [x_{j}(m_{j});y_{i}(n_{i})]_{s,t}$ be a signed $(s,t)$-simplex in a vector space $X$.
We say that $D$ is \textit{balanced} if
\begin{eqnarray*}
\sum\limits_{j} m_{j}x_{j} & = & \sum\limits_{i} n_{i}y_{i}.
\end{eqnarray*}
\end{definition}

Notice that if $D = [x_{j}(m_{j});y_{i}(n_{i})]_{s,t}$ is a degenerate simplex in a vector space $X$, then
$D$ is balanced by Remark \ref{snifter}.

\begin{lemma}\label{B}
Let $D_{1}$ and $D_{2}$ be signed simplices in a vector space $X$.
If $D_{1}$ is balanced and if $D_{1} \succ D_{2}$, then $D_{2}$ is balanced.
(Any refinement of a balanced simplex is balanced.)
\end{lemma}

\begin{proof}
Immediate from the definitions.
\end{proof}

Non-degenerate balanced simplices have important structural properties such as the following.

\begin{theorem}\label{thm:B}
Let $n \geq 1$ be an integer and let $X$ be a real or complex vector space. Then a subset
$Z = \{ z_{0}, z_{1}, \ldots z_{n} \}$ of $X$ admits a non-degenerate balanced simplex if and
only if the set $\{ z_{1} - z_{0}, z_{2} - z_{0}, \ldots, z_{n} - z_{0} \}$ is linearly
dependent (when $X$ is considered as a real vector space).
\end{theorem}

\begin{proof} ($\Rightarrow$) Suppose $Z$ admits a a non-degenerate balanced simplex $D = [x_{j}(m_{j});y_{i}(n_{i})]_{s,t}$.
By Lemma \ref{B} and Lemma \ref{SL1}, we may assume that the simplex $D$ is completely refined.
By definition, we have $\sum m_{j}x_{j} = \sum n_{i}y_{i}$ and $\sum m_{j} = \sum n_{i}$
with at least one (and, in fact, all) $m_{j} \not= 0$. By relabeling the elements of $Z$,
if necessary, we may assume that $x_{1} = z_{0}$. Since $m_{1} = (n_{1} + \cdots + n_{t}) - (m_{2} + \cdots + m_{s})$,
we see that
\begin{eqnarray*}
\bigl( (n_{1} + \cdots + n_{t}) - (m_{2} + \cdots + m_{s})\bigl)z_{0} & = &
\sum\limits_{i} n_{i}y_{i} - \sum\limits_{j \geq 2} m_{j}x_{j}.
\end{eqnarray*}
In other words,
\begin{eqnarray*}
0 & = & \sum\limits_{i} n_{i}(y_{i} - z_{0}) - \sum\limits_{j \geq 2} m_{j}(x_{j} - z_{0}).
\end{eqnarray*}
This shows that the set $\{ z_{1} - z_{0}, z_{2} - z_{0}, \ldots, z_{n} - z_{0} \}$ has a non-empty linearly
dependent subset. Hence the set $\{ z_{1} - z_{0}, z_{2} - z_{0}, \ldots, z_{n} - z_{0} \}$ is linearly dependent.

Now suppose that the set $\{ z_{1} - z_{0}, z_{2} - z_{0}, \ldots, z_{n} - z_{0} \}$ linearly dependent
(when $X$ is considered as a real vector space).
Then there exist real numbers $c_{1}, \ldots, c_{n}$, not all $0$, such that
$$c_{1}(z_{1} - z_{0}) + c_{2}(z_{2} - z_{0}) + \cdots + c_{n}(z_{n} - z_{0}) = 0.$$
Setting $c_{0} = - (c_{1} + \cdots + c_{n})$, so that $c_{0} + c_{1} + \cdots + c_{n} = 0$,
we see that $c_{0}z_{0} + c_{1}z_{1} + \cdots + c_{n}z_{n} = 0$. Thus
\begin{eqnarray}\label{d:1}
\sum\limits_{j : c_{j} > 0} c_{j}z_{j}
& = &
\sum\limits_{i : c_{i} \leq 0} - c_{i}z_{i}.
\end{eqnarray}
By construction,
\begin{eqnarray}\label{d:2}
\sum\limits_{j : c_{j} > 0} c_{j}
& = &
\sum\limits_{i : c_{i} \leq 0} - c_{i},
\end{eqnarray}
and we have already stated that not all of the $c$'s are $0$. By discarding any $c$'s that are equal
to $0$, we deduce from (\ref{d:1}) and (\ref{d:2}) that the set $Z = \{ z_{0}, z_{1}, \ldots, z_{n} \}$
admits a non-degenerate balanced simplex $D$.
(One half of the simplex $D$ is $\{ z_{j}(c_{j}) : c_{j} > 0 \}$ and the other half is
$\{ z_{i}(- c_{i}) : c_{i} < 0 \}$.)
\end{proof}

The following theorem is a variation on a theme developed in Lennard \textit{et al}.\ \cite{Ltw}:
Enflo's \cite{En2} formulation of the generalized roundness of a metric space $(X,d)$ coincides
with $\sup \{ p : (X,d) \text{ has } p\text{-negative type} \}$. This theme was also explored by
Doust and Weston \cite{Dou} within the framework of strict negative type. The new ingredient in
the following theorem is allowing simplices to include possibly negative weights on the vertices.

\begin{theorem}\label{2.3}
Let $p \geq 0$ and let $(X,d)$ be a metric space. Then the following conditions are equivalent:
\begin{enumerate}
\item[(1)] $(X,d)$ has $p$-negative type.

\item[(2)] $\gamma_{p}(D) \geq 0$ for each signed simplex $D$ in $X$.
\end{enumerate}
Moreover, $(X,d)$ has strict $p$-negative type if and only if $\gamma_{p}(D) > 0$ for each
non-degenerate signed simplex $D$ in $(X,d)$.
\end{theorem}

\begin{proof} $(1) \Rightarrow (2)$. Suppose (1) holds.
Let $D = [x_{j}(m_{j});y_{i}(n_{i})]_{s,t}$ be a given signed $(s,t)$-simplex in $X$.
If $D$ is degenerate, then $\gamma_{p}(D) = 0$ by Corollary \ref{0}.
Assume $D$ is non-degenerate. By Lemma
\ref{SL1}, we may further assume that $D$ is completely refined. Let $n = s+t$.
For $1 \leq j \leq s$, set $z_{j} = x_{j}$ and $\alpha_{j} = m_{j}$.
For $1 \leq i \leq t$, set $z_{s+i} = y_{i}$ and $\alpha_{s+i} = - n_{i}$.
Since $m_{1} + \cdots + m_{s} = n_{1} + \cdots + n_{t}$, we see that $\alpha_{1} + \cdots + \alpha_{n} = 0$.
Moreover, $\{ z_{1}, \ldots, z_{n} \}$ is a subset of $X$ (no repetitions)
and each $\alpha_{k} \not= 0$ because the simplex
$D$ is completely refined. It is then a relatively simple matter to check that
\begin{eqnarray}\label{gap:=}
\sum\limits_{j,i} d(z_{j}, z_{i})^{p} \alpha_{j}\alpha_{i}
& = & - 2 \cdot \gamma_{p}(D).
\end{eqnarray}
Therefore $\gamma_{p}(D) \geq 0$. Moreover, $\gamma_{p}(D) > 0$ if $(X,d)$ has strict $p$-negative type.

$(2) \Rightarrow (1)$. Suppose (2) holds.
Let $\{ z_{1}, \ldots, z_{n} \}$ be a given non-empty finite subset of $X$.
Let $\alpha_{1}, \ldots, \alpha_{n}$ be a given collection of real numbers that satisfy
$\alpha_{1} + \cdots + \alpha_{n} = 0$. To avoid triviality, we may assume that not all
of the $\alpha_{k}$'s are $0$. (This forces $n \geq 2$.)
By relabeling $z_{1}, \ldots, z_{n}$, if necessary, we may choose integers $s,t > 0$ such that $s + t = n$,
$\alpha_{1}, \alpha_{2},  \ldots, \alpha_{s} \geq 0$ and $\alpha_{s+1}, \alpha_{s+2},  \ldots, \alpha_{n} < 0$.
Notice that $\alpha_{1} + \cdots + \alpha_{s} = -(\alpha_{s+1} + \cdots + \alpha_{n}) > 0$. Now set
$x_{j} = z_{j}$ and $m_{j} = \alpha_{j}$ for all $j$, $1 \leq j \leq s$. Similarly, set $y_{i} = z_{s+i}$
and $n_{i} = -\alpha_{s+i}$ for all $i$, $1 \leq i \leq t$. By construction,
$D = [x_{j}(m_{j});y_{i}(n_{i})]_{s,t}$ is a non-degenerate signed $(s,t)$-simplex in $X$. By assumption,
$\gamma_{p}(D) \geq 0$. However, the $p$-simplex gap of $D$ also satisfies (\ref{gap:=}), thus:
\begin{eqnarray}\label{gap:s}
\sum\limits_{j,i} d(z_{j}, z_{i})^{p} \alpha_{j}\alpha_{i}
& \leq & 0.
\end{eqnarray}
Moreover, if the inequality in (2) is strict for each non-degenerate signed simplex in $X$, then the
inequality (\ref{gap:s}) will be strict.
\end{proof}

As an immediate application of Theorem \ref{2.3} we obtain the converse of Corollary \ref{0}.

\begin{corollary}\label{cor2.3}
Let $D$ be a signed $(s,t)$-simplex in a metric space $(X,d)$. If $\gamma_{p}(D) = 0$ for all $p \geq 0$,
then $D$ is degenerate.
\end{corollary}

\begin{proof} We prove the contrapositive. Suppose $D$ is non-degenerate.
By Remark \ref{SR5} (3) and Lemma \ref{SL1} we may assume that $D$ is completely refined. Now, $(S(D),d)$
is a finite metric space and thus has strict $p$-negative type for some $p > 0$. In particular,
$\gamma_{p}(D) > 0$ by Theorem \ref{2.3}. This completes the proof.
\end{proof}

We are now in a position to rigorously formulate the notion of a non-trivial $p$-polygonal equality.

\begin{definition}\label{S6}
Let $p \geq 0$ and let $(X,d)$ be a metric space.
A \textit{$p$-polygonal equality} in $(X, d)$ is an equality of the form $\gamma_{p}(D) = 0$
where $D$ is a signed simplex in $X$. If, moreover, the underlying simplex $D$ is non-degenerate, we will say that
the $p$-polygonal equality is \textit{non-trivial}.
\end{definition}

The motivation for defining a non-trivial $p$-polygonal equality is clearly evident from Theorem \ref{2.3}.
In fact, Theorem \ref{2.3} implies the following useful lemmas.

\begin{lemma}\label{s:subsets}
Let $p > 0$ and let $(X,d)$ be a metric space that has $p$-negative type.
Then $(X,d)$ has strict $p$-negative type if and only if it admits no non-trivial $p$-polygonal equality.
\end{lemma}

\begin{proof}
Immediate from Theorem \ref{2.3}.
\end{proof}

\begin{lemma}\label{r:line}
Let $(X,d)$ be a metric space whose generalized roundness $\wp$ is non-zero and suppose that
$0 \leq p < \wp$. Then $(X,d)$ admits no non-trivial $p$-polygonal equalities.
\end{lemma}

\begin{proof}
$(X,d)$ has strict $p$-negative type by Theorem \ref{LW1}.
Now apply Lemma \ref{s:subsets}.
\end{proof}

\begin{lemma}\label{s:equality}
Let $p > 0$. If a metric space $(X,d)$ admits a non-trivial $p$-polygonal equality, then:
\begin{enumerate}
\item $(X,d)$ does not have $q$-negative type for any $q > p$, and
\item $(X,d)$ does not have strict $p$-negative type.
\end{enumerate}
\end{lemma}

\begin{proof}
Suppose $(X,d)$ admits a non-trivial $p$-polygonal equality for some $p > 0$.
If we assume that $(X,d)$ has $q$-negative type for some $q > p$, then it must
have strict $p$-negative type by Theorem \ref{LW1}. However,
this would contradict Lemma \ref{s:subsets}.
\end{proof}

Properties of strict negative type and Lemma \ref{s:equality} imply the following non-embedding principle:

\begin{theorem}\label{iso:emb2}
Let $(Y, \rho)$ be a metric space whose generalized roundness $\wp$ is finite.
Let $(X,d)$ be a metric space that has strict $q$-negative type for some $q \geq \wp$. Then,
$(X,d)$ is not isometric to any metric subspace of $(Y, \rho)$ that admits a non-trivial $p$-polygonal
equality for some $p$ such that $\wp \leq p \leq q$.
\end{theorem}

\begin{proof}
Let $Z \subseteq Y$. Suppose that $(Z, \rho)$ admits a non-trivial $p$-polygonal equality for
some $p \in [\wp, q]$. By Lemma \ref{s:equality}, $(Z, \rho)$ does not have strict $p$-negative type.
On the other hand, $(X,d)$ has strict $p$-negative type because $p \leq q$. (This is a consequence
of Theorem \ref{LW1}.) Hence $(X,d)$ is not isometric to $(Z, \rho)$.
\end{proof}

\begin{remark}\label{iso:emb1}
If, in the statement of Theorem \ref{iso:emb2}, it is the case that $\wp \leq p < q$, then it suffices to
assume that the metric space $(X,d)$ has $q$-negative type. We will then have $\wp(Z) \leq p$ and $\wp(X) \geq q$.
\end{remark}

\section{Polygonal equalities and virtual degeneracy in $L_{p}$-spaces}\label{sec:4}

In order to apply Theorem \ref{iso:emb2} we turn our attention to the study of $p$-polygonal equalities in
$L_{p}$-spaces, $0 < p < \infty$. Our starting point is the following useful consequence of Corollary \ref{0} and
Lemma \ref{r:line}.

\begin{theorem}\label{r:simplex}
Let $(X,d)$ be a metric space whose generalized roundness $\wp$ is non-zero and suppose that $0 \leq p < \wp$.
Given a signed $(s,t)$-simplex $D = [x_{j}(m_{j});y_{i}(n_{i})]_{s,t}$ in $X$,
we have $\gamma_{p}(D) = 0$ if and only if the simplex $D$ is degenerate.
\end{theorem}

\begin{proof}
$(\Rightarrow)$ Suppose $0 \leq p < \wp$ and that $\gamma_{p}(D) = 0$. By Lemma \ref{r:line}, $(X,d)$
admits no non-trivial $p$-polygonal equalities. Hence $D$ must be degenerate.

$(\Leftarrow)$ An immediate consequence of Corollary \ref{0}.
\end{proof}

It is worth recalling that the backward implication in the statement of Theorem \ref{r:simplex} holds for all $p > 0$.

The complex plane endowed with the usual metric has generalized roundness $\wp = 2$.
This leads to the following special case of Theorem \ref{r:simplex} that will be used in the proof of
Theorem \ref{vd:simplex}.

\begin{corollary}\label{cor:simplex}
Let $0 \leq p < 2$.
Given a signed $(s,t)$-simplex $D = [x_{j}(m_{j});y_{i}(n_{i})]_{s,t}$ in the complex plane endowed
with the usual metric, we have $\gamma_{p}(D) = 0$ if and only if the simplex $D$ is degenerate.
\end{corollary}

\begin{definition}\label{v:degenerate}
Let $0 < p < \infty$ and suppose that $(\Omega, \mu)$ is a measure space.
A non-degenerate signed $(s,t)$-simplex $D = [x_{j}(m_{j});y_{i}(n_{i})]_{s,t}$ in $L_{p}(\Omega, \mu)$ is said to be
\textit{virtually degenerate} if the family of signed $(s,t)$-simplices $D(\omega) = [x_{j}(\omega)(m_{j});y_{i}(\omega)(n_{i})]_{s,t}$,
$\omega \in \Omega$, are degenerate in the scalar field of $L_{p}(\Omega, \mu)$ $\mu$-a.e.
\end{definition}

Examples of virtually degenerate simplices are constructed in the proofs of Lemma \ref{vd:ex}, Lemma \ref{vds:lem},
Theorem \ref{inf:vds} and Remark \ref{vds:rem}.

\begin{lemma}\label{vd:bal}
Let $0 < p < \infty$ and suppose that $(\Omega, \mu)$ is a measure space.
Let $D = [x_{j}(m_{j});y_{i}(n_{i})]_{s,t}$ be a signed $(s,t)$-simplex in $L_{p}(\Omega, \mu)$.
If $D$ is virtually degenerate, then $D$ is balanced.
\end{lemma}

\begin{proof}
By Remark \ref{snifter} and the definition of virtual degeneracy we have $$\sum\limits_{j} m_{j}x_{j}(\omega)
= \sum\limits_{i} n_{i}y_{i}(\omega)$$ for almost all $\omega \in \Omega$. Now integrate to get the
desired conclusion.
\end{proof}

The converse of Lemma \ref{vd:bal} is not true in general; indeed, consider the following points
$x_{1} = (0,0), y_{1} = (1,1), x_{2} = (3,1), y_{2} = (2,0)$ in $\ell_{p}^{(2)}$. The non-degenerate signed
$(2,2)$-simplex $D = [x_{j}(1);y_{i}(1)]_{2,2}$ satisfies $x_{1} + x_{2} = y_{1} + y_{2}$ but it is not
virtually degenerate.

In the case of $\ell_{p}^{(n)}$ as well as $\ell_{p}$ the condition that defines virtual degeneracy will hold everywhere.
There are other settings where this may occur but we will not discuss them here. The importance of virtually
degenerate simplices in $L_{p}(\Omega, \mu)$ is that they give rise to a large class of non-trivial $p$-polygonal equalities.

\begin{lemma}\label{vd:lemma}
Let $0 < p < \infty$ and suppose that $(\Omega, \mu)$ is a measure space.
If $D = [x_{j}(m_{j});y_{i}(n_{i})]_{s,t}$ is a virtually degenerate simplex
in $L_{p}(\Omega, \mu)$, then $\gamma_{p}(D) = 0$. In other words, we have the non-trivial $p$-polygonal equality
\begin{eqnarray*}
\sum\limits_{j_{1} < j_{2}} m_{j_{1}}m_{j_{2}} {\| x_{j_{1}} - x_{j_{2}}\|}_{p}^{p} +
\sum\limits_{i_{1} < i_{2}} n_{i_{1}}n_{i_{2}} {\| y_{i_{1}} - y_{i_{2}}\|}_{p}^{p}
&   =  & \sum\limits_{j, i} m_{j}n_{i} {\| x_{j} - y_{i} \|}_{p}^{p}.
\end{eqnarray*}
\end{lemma}

\begin{proof}
If we assume that $D = [x_{j}(m_{j});y_{i}(n_{i})]_{s,t}$ is a virtually degenerate simplex in $L_{p}(\Omega, \mu)$
then, by definition, we have
\begin{eqnarray*}
\sum\limits_{j_{1} < j_{2}} m_{j_{1}}m_{j_{2}} | x_{j_{1}}(\omega) - x_{j_{2}}(\omega) |^{p} +
\sum\limits_{i_{1} < i_{2}} n_{i_{1}}n_{i_{2}} | y_{i_{1}}(\omega) - y_{i_{2}}(\omega) |^{p}
&   =  & \sum\limits_{j, i} m_{j}n_{i} | x_{j}(\omega) - y_{i}(\omega) |^{p}
\end{eqnarray*}
for almost all $\omega \in \Omega$. Integrating over $\Omega$ with respect to $\mu$ gives the desired conclusion.
\end{proof}

\begin{theorem}\label{vd:simplex}
Let $0 < p < 2$ and suppose that $(\Omega, \mu)$ is a measure space.
Given a non-degenerate signed $(s,t)$-simplex $D = [x_{j}(m_{j});y_{i}(n_{i})]_{s,t}$ in
$L_{p}(\Omega, \mu)$, we have the non-trivial $p$-polygonal equality
\begin{eqnarray}\label{three}
\sum\limits_{j_{1} < j_{2}} m_{j_{1}}m_{j_{2}} {\| x_{j_{1}} - x_{j_{2}}\|}_{p}^{p} +
\sum\limits_{i_{1} < i_{2}} n_{i_{1}}n_{i_{2}} {\| y_{i_{1}} - y_{i_{2}}\|}_{p}^{p}
&   =  & \sum\limits_{j, i} m_{j}n_{i} {\| x_{j} - y_{i} \|}_{p}^{p}
\end{eqnarray}
if and only if the simplex is virtually degenerate.
\end{theorem}

\begin{proof}
$(\Rightarrow)$
It suffices to assume that the scalar field of $L_{p}(\Omega, \mu)$ is the complex plane $\mathbb{C}$.
Let $D = [x_{j}(m_{j});y_{i}(n_{i})]_{s,t}$ be a given non-degenerate signed
$(s,t)$-simplex in $L_{p}(\Omega, \mu)$ for which the equality (\ref{three}) holds.
Then $D(\omega) = [x_{j}(\omega)(m_{j});y_{i}(\omega)(n_{i})]_{s,t}$ is a signed $(s,t)$-simplex
in the complex plane for each $\omega \in \Omega$. Moreover, as the complex plane endowed with the
usual metric has $p$-negative type, it follows that we have $\gamma_{p}(D(\omega)) \geq 0$ for
each $\omega \in \Omega$ by Theorem \ref{2.3}. In other words,
\begin{eqnarray}\label{four}
\sum\limits_{j_{1} < j_{2}} m_{j_{1}}m_{j_{2}} | x_{j_{1}}(\omega) - x_{j_{2}}(\omega) |^{p} +
\sum\limits_{i_{1} < i_{2}} n_{i_{1}}n_{i_{2}} | y_{i_{1}}(\omega) - y_{i_{2}}(\omega) |^{p}
& \leq  & \sum\limits_{j, i} m_{j}n_{i} | x_{j}(\omega) - y_{i}(\omega) |^{p}
\end{eqnarray}
for each $\omega \in \Omega$. The inequalities (\ref{four}) cannot be strict on any set of positive
measure as this would imply
\begin{eqnarray*}
\sum\limits_{j_{1} < j_{2}} m_{j_{1}}m_{j_{2}} {\| x_{j_{1}} - x_{j_{2}}\|}_{p}^{p} +
\sum\limits_{i_{1} < i_{2}} n_{i_{1}}n_{i_{2}} {\| y_{i_{1}} - y_{i_{2}}\|}_{p}^{p}
& <  & \sum\limits_{j, i} m_{j}n_{i} {\| x_{j} - y_{i} \|}_{p}^{p},
\end{eqnarray*}
thereby violating (\ref{three}). So the inequalities (\ref{four}) must
hold at equality $\mu$-a.e.\ on $\Omega$. Therefore the family of signed $(s,t)$-simplices
$D(\omega) = [x_{j}(\omega)(m_{j});y_{i}(\omega)(n_{i})]_{s,t}$, $\omega \in \Omega$, are degenerate in the
complex plane $\mu$-a.e. by the forward implication of Corollary \ref{cor:simplex}.
In other words, the simplex $D$ is virtually degenerate.

$(\Leftarrow)$ This follows immediately from Lemma \ref{vd:lemma}.
\end{proof}

It is notable that the forward implication of Theorem \ref{vd:simplex} does not hold for $p =2$:
Every parallelogram in a Hilbert space gives rise to a non-trivial $2$-polygonal equality due to the
\textit{parallelogram identity} but not all parallelograms are virtually degenerate. In the following
section, Theorem \ref{thm:H}
gives a complete description of the $2$-polygonal equalities in any real or complex inner product space.

The following lemma deals with degenerate simplices that have weight $1$ on each vertex.
This leads to a considerably more general form of Theorem \ref{elsner} (Elsner \textit{et al}.\
\cite[Theorem 2.3]{Els}).

\begin{lemma}\label{d:perm}
Let $n \geq 1$ be an integer. A signed $(n,n)$-simplex of the form $D = [x_{j}(1); y_{i}(1)]_{n,n}$
in a metric space $(X,d)$ is degenerate if and only if there exists a permutation $\pi(k)$ of
$(1,2, \ldots, n)$ such that $x_{\pi(k)} = y_{k}$ for each $k, 1 \leq k \leq n$.
\end{lemma}

\begin{proof}
($\Rightarrow$) Suppose that the simplex $D = [x_{j}(1); y_{i}(1)]_{n,n}$ in $X$ is degenerate.
Let $S(D)$ denote the set of distinct points in $X$ that appear in $D$. Say, $S(D) = \{ z_{1}, \ldots, z_{l} \}$.
Because each vertex $x_{j}$ or $y_{i}$ in $D$ has weight $1$ we see that ${\mathbf{m}}(z_{k}) = |\{ j : x_{j} = z_{k} \}|$
and ${\mathbf{n}}(z_{k}) = |\{ i : y_{i} = z_{k} \}|$ for each $k, 1 \leq k \leq l$. For notational simplicity,
we set $m_{k} = {\mathbf{m}}(z_{k})$ and $n_{k} = {\mathbf{n}}(z_{k})$ for each $k, 1 \leq k \leq l$. The
assumption on $D$ is that $m_{k} = n_{k}$ for each $k, 1 \leq k \leq l$. By additionally setting $m_{0} = n_{0} = 0$,
we may choose permutations $\phi(k)$ and $\sigma(k)$ of $(1,2, \ldots, n)$ so that
$z_{k} = x_{\phi(m_{0} + \cdots + m_{k-1} + 1)} = \cdots = x_{\phi(m_{0} + \cdots + m_{k})}
= y_{\sigma(n_{0} + \cdots + n_{k-1} + 1)} = \cdots = y_{\sigma(n_{0} + \cdots + n_{k})}$
for each $k, 1 \leq k \leq l$. (Points from each half of the simplex that are equal are now
arranged in blocks of equal length.) It follows from our construction that $x_{\phi(k)} = y_{\sigma(k)}$
for each $k, 1 \leq k \leq n$. All that remains is to define the permutation $\pi = \phi \sigma^{-1}$.
We then have $x_{\pi(k)} = y_{k}$ for each $k, 1 \leq k \leq n$, as asserted.

($\Leftarrow$) Suppose there is a permutation $\pi(k)$ of $(1,2, \ldots, n)$ such that $x_{\pi(k)} = y_{k}$
for each $k, 1 \leq k \leq n$. Let $z \in S(D)$ be given. Once again it is the case that
${\mathbf{m}}(z) = |\{ j : x_{j} = z \}|$ and ${\mathbf{n}}(z) = |\{ i : y_{i} = z \}|$.
Because of the assumption on $D$ we see that if $z = y_{i}$, then $z = x_{j}$ where $j = \pi(i)$.
Hence ${\mathbf{n}}(z) \leq {\mathbf{m}}(z)$. However, it is also the case that $x_{k} = y_{\pi^{-1}(k)}$
for each $k, 1 \leq k \leq n$. So, by the analogous argument, ${\mathbf{m}}(z) \leq {\mathbf{n}}(z)$ too.
In a nutshell, ${\mathbf{m}}(z) = {\mathbf{n}}(z)$. We conclude that $D$ is degenerate.
\end{proof}

\begin{corollary}\label{els:cor0}
Let $0 < p < 2$ and suppose that $(\Omega, \mu)$ is a measure space.
Let $x_{1}, \ldots, x_{n}, y_{1}, \ldots, y_{n}$ be given functions in $L_{p}(\Omega, \mu)$
such that the signed $(n,n)$-simplex $D = [x_{j}(1);y_{i}(1)]_{n,n}$ is non-degenerate. Then
we have the non-trivial $p$-polygonal equality
\begin{eqnarray}\label{els:pee}
\sum\limits_{j_{1} < j_{2}} {\| x_{j_{1}} - x_{j_{2}} \|}_{p}^{p} +
\sum\limits_{i_{1} < i_{2}} {\| y_{i_{1}} - y_{i_{2}} \|}_{p}^{p}
& = &
\sum\limits_{j, i = 1}^{n} {\| x_{j} - y_{i} \|}_{p}^{p}
\end{eqnarray}
if and only if for almost every $\omega \in \Omega$, the numerical sets
$\{ x_{k}(\omega) \}_{k=1}^{n}$ and $\{ y_{k}(\omega) \}_{k=1}^{n}$ are identical.
\end{corollary}

\begin{proof} By Theorem \ref{vd:simplex}, the equality (\ref{els:pee}) holds if and only if the simplex
$D = [x_{j}(1);y_{i}(1)]_{n,n}$ is virtually degenerate. Now let $\omega \in \Omega$ be given. By Lemma \ref{d:perm},
the simplex $D(\omega) = [x_{j}(\omega)(1);y_{i}(\omega)(1)]_{n,n}$ is degenerate if and only if
there exists a permutation $\pi(\omega, k)$ of $(1,2, \ldots, n)$ such that
$x_{\pi(\omega, k)}(\omega) = y_{k}(\omega)$ for each $k, 1 \leq k \leq n$. In other words,
the simplex $D(\omega) = [x_{j}(\omega)(1);y_{i}(\omega)(1)]_{n,n}$ is degenerate if and only if
the numerical sets $\{ x_{k}(\omega) \}_{k=1}^{n}$ and $\{ y_{k}(\omega) \}_{k=1}^{n}$ are identical.
The corollary is now evident.
\end{proof}

By specializing Corollary \ref{els:cor0} to the case $p = 1$ we obtain Theorem \ref{elsner}
(Elsner \textit{et al}.\ \cite[Theorem 2.3]{Els}).

Lemma \ref{vd:lemma} and Theorem \ref{vd:simplex} have a number of other interesting corollaries. The first is a
classification of the subsets of $L_{p}$-spaces ($0 < p < 2$) that have strict $p$-negative type in terms of virtual degeneracy.

\begin{corollary}\label{Lp}
Let $0 < p < 2$ and suppose that $(\Omega, \mu)$ is a measure space.
A non-empty subset of $L_{p}(\Omega, \mu)$ has strict $p$-negative type if and only if it does not
admit any virtually degenerate simplices.
\end{corollary}

\begin{proof}
This is an immediate consequence of Lemma \ref{s:subsets} and Theorem \ref{vd:simplex}.
\end{proof}

\begin{corollary}\label{aff:Lp}
Let $0 < p < 2$ and suppose that $(\Omega, \mu)$ is a measure space.
If $Z$ is a non-empty subset of $L_{p}(\Omega, \mu)$ that does not have strict $p$-negative type,
then $Z$ is an affinely dependent subset of $L_{p}(\Omega, \mu)$ (when $L_{p}(\Omega, \mu)$ is
considered as a real vector space).
The converse statement is not true in general.
\end{corollary}

\begin{proof}
By Lemma \ref{s:subsets}, $Z$ admits a non-trivial $p$-polygonal equality. So $Z$ admits a virtually
degenerate simplex $D = [x_{j}(m_{j}); y_{i}(n_{i})]_{s,t}$ by Theorem \ref{vd:simplex}. As before
(with a slight abuse of notation), let $S(D) = \{ x_{j}, y_{i} \}$. The simplex $D$ is non-degenerate
by definition of virtual degeneracy and balanced by Lemma \ref{vd:bal}. By Theorem \ref{thm:B},
$S(D)$, and hence $Z$, is an affinely dependent subset of $X = L_{p}(\Omega, \mu)$.

To see that the converse statement is not true in general, consider the points
$z_{0} =(0,0), z_{1} =(1,1), z_{2} =(3,1), z_{3} = (2,0) \in \ell_{p}^{(2)}$. The
set $Z = \{ z_{0}, z_{1}, z_{2}, z_{3} \} \subset \ell_{p}^{(2)}$ is affinely dependent
but it does not admit any virtually degenerate simplices. Hence $Z$ has strict $p$-negative type
by Corollary \ref{Lp}.
\end{proof}

The proof of the following lemma indicates that virtually degenerate simplices are easily
constructed in $L_{p}$-spaces. We will see that this has a number of interesting ramifications.

\begin{lemma}\label{vd:ex}
Let $0 < p < \infty$ and suppose that $(\Omega, \mu)$ is a measure space.
\begin{enumerate}
\item Any linear subspace of $L_{p}(\Omega, \mu)$ that contains a pair of disjointly supported non-zero vectors admits
a virtually degenerate simplex.

\item Every open ball in $L_{p}(\Omega, \mu)$ admits a virtually degenerate simplex.
\end{enumerate}
\end{lemma}

\begin{proof} Suppose that $0 < p < \infty$. Consider two disjointly supported non-zero vectors
$u, v \in L_{p}(\Omega, \mu)$. Set $x_{1} = 0, x_{2} = u + v, y_{1} = u$ and $y_{2} = v$. Then
it is easy to check that $D = [x_{1}(1),x_{2}(1); y_{1}(1),y_{2}(1)]_{2,2}$ is a virtually degenerate
simplex in $L_{p}(\Omega, \mu)$. Any linear subspace of $L_{p}(\Omega, \mu)$ that contains $u$ and $v$ also contains $D$.
This establishes (1). On the other hand, every open ball in $L_{p}(\Omega, \mu)$
contains a translate of a dilation or contraction of the virtually degenerate simplex $D$. These operations
preserve virtual degeneracy. Hence every open ball in $L_{p}(\Omega, \mu)$ contains a virtually degenerate simplex.
This establishes (2).
\end{proof}

\begin{remark}
The condition on the vectors appearing in the statement of Lemma \ref{vd:ex} is well understood when $p \geq 1$.
Indeed, it is germane to recall the following basic fact about $L_{p}$-spaces, $p \not= 2$. Let $1 \leq p < 2$
or $2 < p < \infty$. Then, vectors $u, v \in L_{p}(\Omega, \mu)$ are disjointly supported if and only if
\begin{eqnarray}\label{22:rem}
{\| u + v \|}_{p}^{p} + {\| u - v \|}_{p}^{p} & = & 2 \left( {\| u \|}_{p}^{p} + {\| v \|}_{p}^{p} \right).
\end{eqnarray}
Note that (\ref{22:rem}) is the $p$-polygonal equality that arises from the simplex $D$ in the proof
of Lemma \ref{vd:ex}.
\end{remark}

The following definition is motivated by Lemma \ref{vd:ex} (1).

\begin{definition}\label{p:enflo}
Let $0 < p < \infty$ and suppose that $(\Omega, \mu)$ is a measure space. A linear subspace $W$ of $L_{p}(\Omega, \mu)$
is said to have \textit{Property E} if there exists a pair of disjointly supported non-zero vectors $u, v \in W$.
\end{definition}

Linear subspaces of $L_{p}$-spaces that have Property E cannot have strict $p$-negative type.

\begin{corollary}\label{open}
Let $0 < p < \infty$ and suppose that $(\Omega, \mu)$ is a measure space.
\begin{enumerate}
\item No linear subspace of $L_{p}(\Omega, \mu)$ with Property E has strict $p$-negative type.

\item No non-empty open subset of $L_{p}(\Omega, \mu)$ has strict $p$-negative type.
\end{enumerate}
\end{corollary}

\begin{proof}
This follows from Lemmas \ref{s:equality}, \ref{vd:lemma} and \ref{vd:ex}.
\end{proof}

The next observation provides a general isometric embedding principle for $L_{p}$-spaces ($0 < p < \infty$).
We remind the reader that all $L_{p}$-spaces in this paper are assumed to be at least two-dimensional.

\begin{theorem}\label{embed}
Let $0 < p < \infty$ suppose that $(\Omega, \mu)$ is a measure space. Let $(X,d)$ be a metric space
that has strict $q$-negative type for some $q \geq p$. Then $(X,d)$ is not isometric to any metric subspace
of $L_{p}(\Omega, \mu)$ that admits a virtually degenerate simplex. In particular,
\begin{enumerate}
\item $(X,d)$ is not isometric to any linear subspace $W$ of $L_{p}(\Omega, \mu)$ that has Property E, and

\item $(X,d)$ is not isometric to any metric subspace of $L_{p}(\Omega, \mu)$ that has non-empty interior.
\end{enumerate}
\end{theorem}

\begin{proof}
The generalized roundness of $Y = L_{p}(\Omega, \mu)$ is finite.
Let $Z$ be a metric subspace of $L_{p}(\Omega, \mu)$ that admits a virtually degenerate simplex.
By Lemma \ref{vd:lemma}, $Z$ admits a non-trivial $p$-polygonal equality. As $(X,d)$ is assumed
to have strict $q$-negative for some $q \geq p$, we deduce that $(X,d)$ is not isometric to $Z$
by Theorem \ref{iso:emb2}. Parts (1) and (2) now follow directly from Lemma \ref{vd:ex}.
\end{proof}

\begin{remark}
For any $p \in (0, \infty)$ there are always metric spaces which have strict $q$-negative type for
some $q \geq p$. For instance, ultrametric spaces have strict $q$-negative type for all $q \geq 0$.
In fact, by Faver \textit{et al}.\ \cite[Theorem 5.2]{Fav}, a metric space $(X,d)$ has $q$-negative type
for all $q \geq 0$ if and only if it is ultrametric.
\end{remark}

The following special case of Theorem \ref{embed} is worth emphasizing.

\begin{corollary}\label{iso:lp}
Let $0 < p < r \leq 2$. Suppose that $(\Omega_{1}, \mu_{1})$ and $(\Omega_{2}, \mu_{2})$ are measure spaces. Then,
no metric subspace of $L_{r}(\Omega_{2}, \mu_{2})$ is isometric to any metric subspace of $L_{p}(\Omega_{1}, \mu_{1})$
that admits a virtually degenerate simplex. In particular,
\begin{enumerate}
\item No subset of $L_{r}(\Omega_{2}, \mu_{2})$ is isometric to any linear subspace $W$ of
$L_{p}(\Omega_{1}, \mu_{1})$ that has Property E.

\item No subset of $L_{r}(\Omega_{2}, \mu_{2})$ is isometric to any subset
of $L_{p}(\Omega_{1}, \mu_{1})$ that has non-empty interior.
\end{enumerate}
\end{corollary}

\begin{proof}
Let $p, r, (\Omega_{1}, \mu_{1})$ and $(\Omega_{2}, \mu_{2})$ be given as in the statement of the corollary.
All non-empty subsets of $L_{r}(\Omega_{2}, \mu_{2})$ have $r$-negative type by Theorem \ref{schoenberg}
and hence strict $p$-negative type by Theorem \ref{LW1}. Now apply Theorem \ref{embed} with $q = p$.
\end{proof}

\section{Polygonal equalities in real and complex inner product spaces}\label{sec:5}

Theorem \ref{vd:simplex} classifies all non-trivial $p$-polygonal equalities in $L_{p}(\Omega, \mu)$,
$0 < p < 2$, according to the notion of virtual degeneracy. In the present section we classify all
non-trivial $2$-polygonal equalities in $L_{2}(\Omega, \mu)$. Theorem \ref{thm:H} illustrates that
there is a marked difference between the cases $p < 2$ and $p = 2$. This is because the real
line does not have strict $2$-negative type. The starting point is the following key lemma.

\begin{lemma}\label{thm:inner}
Let $(X, \langle \cdot , \cdot \rangle)$ be a real or complex inner product space with
induced norm $\| \cdot \|$.
Let $s,t > 0$ be integers. If $D = [x_{j}(m_{j});y_{i}(n_{i})]_{s,t}$ is a signed $(s,t)$-simplex
in $X$, then
\begin{eqnarray}\label{sq:lem2}
{\Biggl\| \sum\limits_{j} m_{j}x_{j} - \sum\limits_{i} n_{i}y_{i} \Biggl\|}^{2} & = &
\sum\limits_{j,i} m_{j}n_{i} {\|x_{j} - y_{i}\|}^{2} \nonumber \\
& & - \sum\limits_{j_{1} < j_{2}} m_{j_{1}}m_{j_{2}} {\|x_{j_{1}} - x_{j_{2}}\|}^{2}
- \sum\limits_{i_{1} < i_{2}} n_{i_{1}}n_{i_{2}} {\|y_{i_{1}} - y_{i_{2}}\|}^{2} \\
& \equiv & \gamma_{2}(D). \nonumber
\end{eqnarray}
\end{lemma}

\begin{proof}
It suffices to assume that $(X, \langle \cdot , \cdot \rangle)$ is a complex inner product space.
(The argument for a real inner product space is entirely similar.) We use $\Re z$ to denote the
real part of a complex number $z \in \mathbb{C}$.

Let $L$ and $R$ denote the left and right sides of (\ref{sq:lem2}), respectively. The claim is that $L = R$.
The first thing to notice is that
\begin{eqnarray}\label{L}
L & = & \left\langle \sum\limits_{j} m_{j}x_{j} - \sum\limits_{i} n_{i}y_{i} ,
                \sum\limits_{j} m_{j}x_{j} - \sum\limits_{i} n_{i}y_{i} \right\rangle \nonumber \\
  & = & \sum\limits_{j} m_{j}^{2} {\| x_{j} \|}^{2}
        + 2 \sum\limits_{j_{1} < j_{2}} m_{j_{1}}m_{j_{2}} \Re \langle x_{j_{1}} , x_{j_{2}} \rangle
        - 2 \sum\limits_{j, i} m_{j}n_{i} \Re \langle x_{j} , y_{i} \rangle \nonumber \\
  &   & + \sum\limits_{i} n_{i}^{2} {\| y_{i} \|}^{2}
        + 2 \sum\limits_{i_{1} < i_{2}} n_{i_{1}}n_{i_{2}} \Re \langle y_{i_{1}} , y_{i_{2}} \rangle.
\end{eqnarray}

On the other hand,
\begin{eqnarray}\label{R}
R & = & \sum\limits_{j, i} m_{j}n_{i} \langle x_{j} - y_{i} , x_{j} - y_{i} \rangle
        - \sum\limits_{j_{1} < j_{2}} m_{j_{1}}m_{j_{2}} \langle x_{j_{1}} - x_{j_{2}} , x_{j_{1}} - x_{j_{2}} \rangle
        - \sum\limits_{i_{1} < i_{2}} n_{i_{1}}n_{i_{2}} \langle y_{i_{1}} - y_{i_{2}} , y_{i_{1}} - y_{i_{2}} \rangle \nonumber \\
  & = & \sum\limits_{j, i} m_{j}n_{i} {\| x_{j} \|}^{2} + \sum\limits_{j, i} m_{j}n_{i} {\| y_{i} \|}^{2}
        - 2 \sum\limits_{j, i} m_{j}n_{i} \Re \langle x_{j} , y_{i} \rangle
        - \sum\limits_{j_{1} < j_{2}} m_{j_{1}}m_{j_{2}} ({\| x_{j_{1}} \|}^{2} + {\| x_{j_{2}} \|}^{2}) \nonumber \\
  &   & + 2 \sum\limits_{j_{1} < j_{2}} m_{j_{1}}m_{j_{2}} \Re \langle x_{j_{1}} , x_{j_{2}} \rangle
        - \sum\limits_{i_{1} < i_{2}} n_{i_{1}}n_{i_{2}} ({\| y_{i_{1}} \|}^{2} + {\| y_{i_{2}} \|}^{2})
        + 2 \sum\limits_{i_{1} < i_{2}} n_{i_{1}}n_{i_{2}} \Re \langle y_{i_{1}} , y_{i_{2}} \rangle.
\end{eqnarray}

Comparing the expressions (\ref{L}) and (\ref{R}) for $L$ and $R$ we see that the derivation of (\ref{sq:lem2})
will be complete if we can derive the following identity.

\begin{eqnarray}\label{id:sharp}
m_{1}^{2}{\|x_{1}\|}^{2} + \cdots + m_{s}^{2}{\|x_{s}\|}^{2} + n_{1}^{2}{\|y_{1}\|}^{2} + \cdots + n_{t}^{2}{\|y_{t}\|}^{2} & = &
\sum\limits_{j,i} m_{j}n_{i} {\|x_{j}\|}^{2} + \sum\limits_{j,i} m_{j}n_{i} {\|y_{i}\|}^{2} \nonumber \\
& & - \sum\limits_{j_{1} < j_{2}} m_{j_{1}}m_{j_{2}}\bigl( {\|x_{j_{1}}\|}^{2} + {\|x_{j_{2}}\|}^{2} \bigl) \nonumber \\
& & - \sum\limits_{i_{1} < i_{2}} n_{i_{1}}n_{i_{2}}\bigl( {\|y_{i_{1}}\|}^{2} + {\|y_{i_{2}}\|}^{2} \bigl).
\end{eqnarray}

Now let $L^{\sharp}$ and $R^{\sharp}$ denote the left and right sides of (\ref{id:sharp}). Recalling that
$m_{1} + \cdots + m_{s} = n_{1} + \cdots + n_{t}$ we are now in a position to complete the proof. In fact,
\begin{eqnarray*}
R^{\sharp} & = &
\bigl( n_{1} + \cdots + n_{t} \bigl)\bigl( m_{1}{\|x_{1}\|}^{2} + \cdots + m_{s}{\|x_{s}\|}^{2} \bigl)
+ \bigl( m_{1} + \cdots + m_{s} \bigl)\bigl( n_{1}{\|y_{1}\|}^{2} + \cdots + n_{t}{\|y_{t}\|}^{2} \bigl) \\
& & - \sum\limits_{j_{1} < j_{2}} m_{j_{1}}m_{j_{2}}\bigl( {\|x_{j_{1}}\|}^{2} + {\|x_{j_{2}}\|}^{2} \bigl)
- \sum\limits_{i_{1} < i_{2}} n_{i_{1}}n_{i_{2}}\bigl( {\|y_{i_{1}}\|}^{2} + {\|y_{i_{2}}\|}^{2} \bigl) \\
& = & \bigl( m_{1} + \cdots + m_{s} \bigl)\bigl( m_{1}{\|x_{1}\|}^{2} + \cdots + m_{s}{\|x_{s}\|}^{2} \bigl)
+ \bigl( n_{1} + \cdots + n_{t} \bigl)\bigl( n_{1}{\|y_{1}\|}^{2} + \cdots + n_{t}{\|y_{t}\|}^{2} \bigl) \\
& & - \sum\limits_{j_{1} < j_{2}} m_{j_{1}}m_{j_{2}} \bigl( {\|x_{j_{1}}\|}^{2} + {\|x_{j_{2}}\|}^{2} \bigl)
- \sum\limits_{i_{1} < i_{2}} n_{i_{1}}n_{i_{2}} \bigl( {\|y_{i_{1}}\|}^{2} + {\|y_{i_{2}}\|}^{2} \bigl) \\
& = & m_{1}^{2}{\|x_{1}\|}^{2} + \cdots + m_{s}^{2}{\|x_{s}\|}^{2} + n_{1}^{2}{\|y_{1}\|}^{2} + \cdots + n_{t}^{2}{\|y_{t}\|}^{2} \\
& = & L^{\sharp}.
\end{eqnarray*}
\end{proof}

A special case of Lemma \ref{thm:inner} for the Hilbert space $L_{2}[0,1]$ (with $s = t$ and $m_{j} = 1 = n_{i}$ for all $j,i$)
was noted in passing by Enflo \cite{En2} and has also been recorded in the literature by several other authors.

An immediate upshot of Theorem \ref{2.3} and Lemma \ref{thm:inner} is the well-known result that
every inner product space has $2$-negative type. More importantly, for our purposes,
Lemma \ref{thm:inner} leads directly to a complete description of the $2$-polygonal equalities
in any given real or complex inner product space.

\begin{theorem}\label{thm:H}
Let $D = [x_{j}(m_{j}); y_{i}(n_{i})]_{s,t}$ be a signed $(s,t)$-simplex in a real or complex inner product space $X$.
Then the following conditions are equivalent:
\begin{enumerate}
\item $\gamma_{2}(D) = 0$.

\item $\sum\limits_{j} m_{j}x_{j} = \sum\limits_{i} n_{i}y_{i}$.
\end{enumerate}
\end{theorem}

\begin{corollary}\label{cor:H1}
Let $X$ be a real or complex inner product space. A metric subspace $Z$ of $X$ has strict $2$-negative type
if and only if it does not admit any non-degenerate balanced simplices.
\end{corollary}

\begin{proof}
Immediate from Lemma \ref{s:subsets} and Theorem \ref{thm:H}.
\end{proof}

\begin{corollary}\label{cor:H2}
Let $Z$ be a non-empty metric subspace of a real or complex inner product space $X$. Then,
$Z$ has strict $2$-negative type if and only if $Z$ is an affinely independent subset of $X$
(when $X$ is considered as a real vector space).
\end{corollary}

\begin{proof}
It suffices to assume that $Z = \{ z_{0}, z_{1}, \ldots z_{n} \}$ for some integer $n > 0$.
The corollary then follows trivially from Corollary \ref{cor:H1} and Theorem \ref{thm:B}.
\end{proof}

There are a number of interesting ways to apply Corollary \ref{cor:H2}. For example,
let $(X,d)$ be an infinite metric space of cardinality $\psi$. The proof of Theorem 1.2 in Lemin \cite{Lem} may be
easily adapted to establish the following result: If $(X,d)$ has strict $2$-negative type, then $(X,d)$ may be
isometrically embedded into a real inner product space $I^{\psi}$ of Hamel dimension $\psi$. By Corollary \ref{cor:H2},
the isometric image of $(X,d)$ in $I^{\psi}$ must be affinely independent. Moreover, it follows from Corollary
\ref{cor:H2} that $(X,d)$ may not be isometrically embedded into any real inner product space of Hamel dimension
$\sigma < \psi$. Conversely, Corollary \ref{cor:H2} ensures that if $(X,d)$ is isometric to an
affinely independent subset of a real inner product space of Hamel dimension $\psi$, then $(X,d)$ has strict
$2$-negative type. In summary, we have established the following theorem.

\begin{theorem}\label{lemin}
A metric space of infinite cardinality $\psi$ has strict $2$-negative type if and only if it is isometric
to an affinely independent subset of a real inner product space of Hamel dimension $\psi$.
\end{theorem}

It is a fundamental result of Schoenberg \cite{Sc3} that every metric space of (strict) $2$-negative type is
isometric to a metric subspace of some real Hilbert space. Corollary \ref{cor:H2} therefore leads to a version
of Theorem \ref{embed} that is specific to Hilbert spaces.

\begin{theorem}\label{cor:H3}
A metric space has strict $2$-negative type if and only if it is isometric to an affinely independent
subset of some real Hilbert space.
In particular, no metric space $(X,d)$ of strict $2$-negative type is isometric to any affinely dependent
metric subspace of any Hilbert space $H$ (when $H$ is considered as a real vector space).
\end{theorem}

\begin{proof}
The stated equivalence follows directly from \cite[Theorem 1]{Sc3} and Corollary \ref{cor:H2}.

If $Z$ is an affinely dependent metric subspace of a Hilbert space $H$ (when $H$ is considered as a real vector space),
then $Z$ does not have strict $2$-negative type by Corollary \ref{cor:H2}. In particular, $Z$ is not isometric to
any metric space $(X,d)$ that has strict $2$-negative type.
\end{proof}

We remark that in Theorem \ref{lemin} or Theorem \ref{cor:H3} the metric space could, for example, be any ultrametric
space. This is because all ultrametric spaces have infinite generalized roundness by Theorem 5.1 in Faver \textit{et al}.\
\cite{Fav} and thus strict $2$-negative type by Theorem \ref{LW1}.

Theorems \ref{lemin} and \ref{cor:H3} imply characterizations of strict $p$-negative type for all $p$ such
that $0 \leq p \leq 2$. Indeed, suppose that $0 \leq p \leq 2$ and that $d$ is a metric on a set $X$. Then
the so-called \textit{metric transform} $d^{p/2}$ is also a metric on $X$. Moreover, it is plainly evident
that $(X,d)$ has strict $p$-negative type if and only if $(X, d^{p/2})$ has strict $2$-negative type. Thus,
combining Theorem \ref{lemin} and \ref{cor:H3}, we obtain the following corollary.

\begin{corollary}\label{cor:H4}
Suppose $0 \leq p \leq 2$.
\begin{enumerate}
\item[(1)] A metric space $(X,d)$ has strict $p$-negative type if and only if $(X,d^{p/2})$
is isometric to an affinely independent subset of some real Hilbert space.

\item[(2)] A metric space $(X,d)$ of infinite cardinality $\psi$ has strict $p$-negative type if and only if
$(X,d^{p/2})$ is isometric to an affinely independent subset of a real inner product space of Hamel dimension $\psi$.
\end{enumerate}
\end{corollary}

\begin{remark}
There are versions of Theorem \ref{lemin}, Theorem \ref{cor:H3} and Corollary \ref{cor:H4} for \textit{finite}
metric spaces that are due to Faver \textit{et al}.\ \cite{Fav}. In the present work there is no restriction being placed
on the cardinality of the metric space. The techniques developed in this paper are substantially different from those
used in \cite{Fav}.
\end{remark}

In relation to a problem of Lemin \cite{Lem} concerning the isometric embedding of ultrametric spaces into
Banach spaces, Shkarin \cite{Shk} introduced the class $\mathcal{M}$ of all finite metric spaces
$(Z, d)$, $Z = \{ z_{0}, z_{1}, \ldots, z_{n} \}$, which admit an isometric embedding
$\phi : Z \rightarrow \text{(real) }\ell_{2}$ such that the vectors $\{ \phi(z_{k}) - \phi(z_{0}) : 1 \leq k \leq n\}$
are linearly independent. Theorem 1 in \cite{Shk} shows that any metric space in $\mathcal{M}$ admits an isometric
embedding into any infinite-dimensional Banach space. As noted by Shkarin, it follows from the work of Lemin (as well
as several other authors), that the class $\mathcal{M}$ contains all finite ultrametric spaces. However, every finite
metric space of (strict) $2$-negative type admits an isometric embedding into real $\ell_{2}$ by Schoenberg \cite{Sc1}.
Combining this result with Corollary \ref{cor:H2} we obtain a complete description of Shkarin's class $\mathcal{M}$.

\begin{theorem}\label{shkarin}
Shkarin's class $\mathcal{M}$ consists of all finite metric spaces of strict $2$-negative type.
\end{theorem}

Faver \textit{et al}.\ \cite{Fav} have given an independent proof of Theorem \ref{shkarin}. It follows
from \cite[Theorem 1]{Shk} that any finite metric space of strict $2$-negative type may be isometrically
embedded into any infinite-dimensional Banach space. In related work, Funano \cite{Fun} has shown that
every proper ultrametric space may be isometrically embedded into $\ell_{p}$ for any $p \geq 1$.
(We recall that a metric space $(Z,d)$ is \textit{proper} if every closed ball in it is compact.)

The purpose of the next section is to take a closer look at linear subspaces of $L_{p}$-spaces that
admit virtually degenerate simplices.

\section{Virtually degenerate subspaces of $L_{p}$-spaces}\label{sec:6}
Determining exactly which linear subspaces of $L_{p}(\Omega, \mu)$ admit a virtually degenerate simplex
is a moot question. For clarity of exposition it is helpful to make the following definition.

\begin{definition}
Let $0 < p < \infty$ and let $(\Omega, \mu)$ be a measure space. A linear subspace $W$ of $L_{p}(\Omega, \mu)$
will be called \textit{virtually degenerate} if it admits a virtually degenerate simplex.
\end{definition}


\begin{remark}\label{vd:sub}
Notice that if $W$ is a virtually degenerate linear subspace of $L_{p}(\Omega, \mu)$, then:
\begin{enumerate}
\item $W$ does not have $q$-negative type for any $q > p$, and

\item $W$ does not have strict $p$-negative type.
\end{enumerate}
Moreover, provided $0 < p \leq 2$, it is the case that
\begin{enumerate}
\item[(3)] $W$ does have $p$-negative type.
\end{enumerate}
\end{remark}

Of course, no normed linear space has (strict) $q$-negative type for any $q > 2$. This is because
normed linear spaces are mid-point convex. So points (1), (2) and (3) are really only of interest
when $0 < p < 2$.

It is necessarily the case that some $L_{p}$-spaces contain linear subspaces that are not virtually
degenerate. This is evident from the following celebrated theorem of Bretagnolle, Dacunha-Castelle
and Krivine \cite{Bre}.

\begin{theorem}[Bretagnolle \textit{et al}.\ \cite{Bre}]\label{bdk}
Let $0 < p \leq 2$ and let $X$ be a real quasi-normed space. Then $X$ is linearly isometric to a
subspace of some $L_{p}$-space if and only if $X$ has $p$-negative type.
\end{theorem}

For instance, if $0 < p < q \leq 2$, then $L_{q}[0,1]$ is linearly isometric to a subspace $W$
of $L_{p}[0,1]$. As $W$ has $q$-negative type, we see that it cannot be a virtually degenerate linear subspace of
$L_{p}[0,1]$ by Remark \ref{vd:sub}.

Lemma \ref{vd:ex} shows that every linear subspace of $L_{p}(\Omega, \mu)$ with Property E
must be virtually degenerate. However, Property E is rather special and it is by no means necessary
for virtual degeneracy. Lemma \ref{vds:lem} provides one means for constructing virtually
degenerate linear subspaces of $L_{p}(\Omega, \mu)$ that do not necessarily have Property E.
We preface this lemma with some helpful notation.

Suppose that $u, v$ are measurable functions on a measure space $(\Omega, \mu)$. Define
$u[v] = u \cdot \chi_{\supp (v)}$ where $\supp (v)$ denotes the support of $v$. If $u$ and $v$
lie in $L_{p}(\Omega, \mu)$ for some $p \in (0, \infty)$, then $u[v]$ is a well-defined
element of $L_{p}(\Omega, \mu)$.

\begin{lemma}\label{vds:lem}
Let $0 < p < \infty$ and suppose that $(\Omega, \mu)$ is a measure space.
Let $W$ be a linear subspace of $L_p(\Omega,\mu)$ that contains linearly independent vectors $u$
and $v$ such that:
\begin{enumerate}
\item $\supp (u) \cap \supp (v)$ has positive measure, and
\item $u[v]$ and $v[u]$ are linearly dependent.
\end{enumerate}
Then $W$ is virtually degenerate.
\end{lemma}

\begin{proof}
We may choose a non-zero scalar $\kappa$ such that $\kappa u[v]-v[u] = 0$.
Let $x_1 = \kappa u - v$, $x_2 = -\kappa u$, $x_3 = v$, $y_1= v- \kappa u$, $y_2 = \kappa u$ and $y_3 = -v$.
The condition that $u$ and $v$ be linearly independent guarantees that $x_{j} \not= y_{i}$ for $j, i \in \{ 1,2,3 \}$,
so the signed $(3,3)$-simplex we construct will be pure. It can be readily checked that:
\begin{enumerate}
\item[(1)] $x_1(\omega)=y_2(\omega)$, $x_2(\omega)=y_1(\omega)$ and $x_3(\omega)=y_3(\omega)$
for $\omega\in\supp (u) \setminus \supp (v)$,
\item[(2)] $x_1(\omega)=y_1(\omega)$, $x_2(\omega)=y_3(\omega)$ and $x_3(\omega)=y_2(\omega)$
for $\omega\in\supp (u)\cap\supp (v)$, and
\item[(3)] $x_1(\omega)=y_3(\omega)$, $x_2(\omega)=y_2(\omega)$ and $x_3(\omega)=y_1(\omega)$
for $\omega\in\supp (v) \setminus \supp (u)$.
\end{enumerate}
This demonstrates that the simplex $D = [x_j(1);y_i(1)]_{3,3}$ in $W$ is virtually degenerate.
\end{proof}

For an example of a linear subspace without Property E which admits a virtually degenerate simplex as per
Lemma \ref{vds:lem}, consider the subspace $W = \Span \left\{ (1,1,0),(0,1,1) \right\}$ of $\ell_p^{(3)}$
and set $u=(1,1,0)$ and $v=(0,1,1)$. The proof of the next theorem shows that Lemma \ref{vds:lem}
can be applied in infinite-dimensional settings.

\begin{theorem}\label{inf:vds}
If $0 < p < \infty$, then $\ell_{p}$ has an infinite-dimensional virtually degenerate linear subspace $W$ without Property E.
\end{theorem}

\begin{proof} Let $0 < p < \infty$.
We construct an infinite-dimensional linear subspace $W$ of $\ell_{p}$ that does not have Property E but which
satisfies the hypotheses of Lemma \ref{vds:lem}. The construction proceeds in the following manner.
Let $p_{n}$ denote the $n$th prime number and define a vector $x_{n} = (x_{n}(l)) \in \ell_{p}$ by setting:
\begin{eqnarray*}
x_{n}(l) & = &\left\{
     \begin{array}{ll}
       2^{-l} & \mbox{if } p_{n}\, |\, l, \\
       0 & \mbox{otherwise.}
     \end{array}
   \right.
\end{eqnarray*}
Notice that $x_{j}[x_{i}] = x_{i}[x_{j}]$ for all $j, i$ by construction.
Now let $W$ denote the linear subspace of $\ell_{p}$ spanned by the set $S = \{ x_{n} : n \geq 1 \}$.
As $x_{n}$ is the only vector in $S$ whose $p_{n}$th coordinate is non-zero we see that $S$ is linearly
independent and hence $W$ is infinite-dimensional.
Consider two non-zero vectors $x, y \in W$. Say,
\begin{eqnarray*}
x = \sum\limits_{j \in J} \kappa_{j} x_{j} \mbox{ and } y = \sum\limits_{k \in K} \upsilon_{k} x_{k} \in W,
\end{eqnarray*}
where $\kappa_{j}, \upsilon_{k} \not= 0$ for all $j \in J, k \in K$. We claim that $x$ and $y$
do not have disjoint support. Indeed, if there exists an $i \in J \cap K$, then both $x$ and $y$ are non-zero in
the $p_{i}$th coordinate. On the other hand, if $j \in J$, $k \in K$ and $J \cap K = \varnothing$, then
$x(p_{j}p_{k})= \kappa_{j}2^{-p_{j}p_{k}}$ and $y(p_{j}p_{k}) = \upsilon_{k}2^{-p_{j}p_{k}}$ are both non-zero.
Thus the infinite-dimensional linear subspace $W \subset \ell_{p}$ does not have Property E. However,
as we have noted that any two basis vectors $x_i,x_j$ of $W$
satisfy the conditions of Lemma \ref{vds:lem}, we see that $W$ is virtually degenerate.
Note also that $W$ has the stronger property that any finite set of vectors in $W$ have intersecting support.
\end{proof}

\begin{remark}\label{vds:rem}
It is worth noting that the preceding construction may be tweaked so that the resulting infinite-dimensional
linear subspace $W \subset \ell_{p}$ does not satisfy Property E or the condition in Lemma \ref{vds:lem}. As before we let
$p_{n}$ denote the $n$th prime number but this time we define $x_{n} = (x_{n}(l)) \in \ell_{p}$ as follows:

\begin{eqnarray*}
x_n(l) &=&\left\{
\begin{array}{ll}
p_n^{-l} & \mbox{if } p_n\,|\,l,\\
0&\mbox{otherwise.}
\end{array}
\right.
\end{eqnarray*}

Now let $W$ denote the linear subspace of $\ell_p$ spanned by the set $S = \{x_n : n \geq 1 \}$.
As $x_n$ is the only vector in $S$ whose $p_n$th coordinate is non-zero we see that $S$ is linearly independent
and hence $W$ is infinite-dimensional. Consider two linearly independent vectors $x, y \in W$. Say,
\begin{eqnarray*}
x= \sum\limits_{j \in J} \kappa_jx_j \mbox{ and } y=\sum\limits_{k \in K} \upsilon_k x_k \in W,
\end{eqnarray*}
where $\kappa_j,\upsilon_k \not=0$ for all $j \in J$, $k \in K$.
(If $j \notin J$, then we may set $\kappa_{j} = 0$, and so on.)
We claim that $x$ and $y$ have intersecting support, and that
$x[y]$ and $y[x]$ are linearly independent. Indeed, if $x$ and $y$ share identical basis vectors then we need only consider
the coordinates $L = \{p_j\,|\,j \in J\}$. These coordinates lie in the support of $x$ and $y$, and if $x$ and $y$ were linearly
dependent on $L$ we would have $\kappa_j=c\upsilon_j$ for some non-zero constant $c$ and all $j \in J$, and this would
imply that $x$ and $y$ are linearly dependent. If $x$ and $y$ do not share identical basis vectors, then we may assume
without loss of generality that $J \setminus K \not= \varnothing$. If $j \in J \setminus K$ and $k \in K$, then $x$ is zero at
at most one coordinate among $p_jp_k,(p_jp_k)^2,(p_jp_k)^3,\dots$ by the uniqueness of any solution to 
$\kappa_jp_j^{-l}+\kappa_kp_k^{-l}=0$ with respect to $l$.
Moreover, $y$ is non-zero on all of these coordinates. The vectors $x$ and $y$ are linearly independent when restricted to any two 
of these coordinates on which $x$ is non-zero (by the uniqueness of any solution to
$\kappa_jp_j^{-l}+cp_k^{-l}=0$ with respect to $l$),
thereby showing that $x[y]$ and $y[x]$ are linearly independent. It follows
that $W$ does not satisfy the hypotheses of Lemma \ref{vds:lem}.
\end{remark}

We conclude this section with a comment on the special case $p = 2$. It is clear that every linear subspace of
$L_{2}(\Omega, \mu)$ admits a non-degenerate balanced simplex. Therefore no linear subspace of $L_{2}(\Omega, \mu)$
has strict $2$-negative type by Corollary \ref{cor:H1}. (In fact, no linear subspace of any normed space has
strict $2$-negative type by mid-point convexity.) This contrasts nicely with the case $0 < p < 2$ where,
as we have noted, $L_{p}(\Omega, \mu)$ may admit linear subspaces of strict $p$-negative type. For instance,
there are linear subspaces of $L_{p}[0,1]$ ($0 < p < 2$) that are linearly isometric to $L_{2}[0,1]$. Such subspaces
have $2$-negative type and hence strict $p$-negative type by Theorem \ref{LW1}.

\section{Open problems}\label{sec:7}

In the case $p > 2$ the classification of all non-trivial $p$-polygonal equalities in $L_{p}(\Omega, \mu)$ has not been
completely settled by the techniques developed in this paper. This is because one cannot argue on the basis of strict
$p$-negative type for values of $p$ in the range $(2, \infty)$. There are two impediments. One is described in Remark \ref{Kold}.
The second impediment is that Corollary \ref{cor:simplex} does not hold for values of $p$ in the range $(2, \infty)$.
Lemma \ref{vd:lemma} shows that virtual degeneracy is sufficient in the case $p > 2$. We do not know if it is necessary.

For each $p > 0$ let $\mathcal{C}_{p}$ denote the Schatten $p$-class. It is well-known that $\mathcal{C}_{2}$ is a
Hilbert space under the inner product $\langle x,y \rangle = tr(y^{\ast}x)$, $x,y \in \mathcal{C}_{2}$.
Thus the equality (\ref{sq:lem2}) stated
in Lemma \ref{thm:inner} is valid for $\mathcal{C}_{2}$. On the other hand, provided $p \not= 2$, $\wp(\mathcal{C}_{p})
= 0$. This is due to Lennard \textit{et al}.\ \cite{Ltw} in the case $p > 2$, and Dahma and Lennard \cite{Dah} in the case $0 < p < 2$.
It therefore makes sense to ask whether or not $\mathcal{C}_{p}$, $p \not= 2$, admits any non-trivial p-polygonal equalities,
and if so, whether or not they can be classified geometrically.
If, for example, there exists a finite metric subspace $X \subset \mathcal{C}_{p}$ such that $\wp(X) = p$, then
$\mathcal{C}_{p}$ will admit a non-trivial $p$-polygonal equality.
This is a consequence of Theorem \ref{LW2} and it motivates a more general problem. Given a Banach space $B$, determine
all values $p \geq \wp(B)$ such that $\wp(X) = p$ for some finite metric subspace $X \subset B$. It will then follow that
the Banach space $B$ admits a non-trivial $p$-polygonal equality for all such values of $p$.

Section \ref{sec:6} provides a glimpse of the complexity of virtually degenerate subspaces of $L_{p}(\Omega, \mu)$.
The interesting case is $0 < p < 2$ and it would seem to be a worthwhile project to develop new ways to construct or
identify virtually degenerate subspaces of $L_{p}(\Omega, \mu)$. The most challenging problem would appear to be
the development of necessary and sufficient conditions for a linear subspace of $L_{p}(\Omega, \mu)$ to be virtually
degenerate.

\section*{Acknowledgments}
The research presented in this paper was initiated at the 2011 Cornell University \textit{Summer Mathematics Institute}
(SMI) and completed under the auspices of the \textit{Visiting Researcher Programme} at the University of South Africa (UNISA).
The authors would like to thank the Department of Mathematics and the Center for Applied Mathematics at Cornell University
for supporting this project, and the National Science Foundation for its financial support of
the SMI through NSF grant DMS-0739338. We are very grateful for the additional financial support from UNISA,
Canisius College and the University of New South Wales that aided in the completion
of this paper. In addition, we would like to thank Petrus Potgieter, Willem Fouch\'{e}, Ian Doust and Stephen S\'{a}nchez
for their particularly helpful input on (infinitely many) preliminary drafts of this paper.
The last named author extends special thanks to the Australian Catholic University for additional support through their
programme of Honorary professorships.

\bibliographystyle{amsalpha}

\end{document}